\documentclass[11pt]{amsart}
\usepackage{amssymb, latexsym, mathrsfs,   color }
\usepackage[all]{xy}
\usepackage[colorlinks=true, pdfstartview=FitV,linkcolor=blue,citecolor=blue,urlcolor=blue]{hyperref}

    \advance \topmargin by -\headheight
    \advance \topmargin by -\headsep

    \evensidemargin \oddsidemargin
\newtheorem {theorem}    {Theorem}[section]

\newtheorem {lemma}      [theorem]    {Lemma}
\newtheorem {corollary}  [theorem]    {Corollary}
\newtheorem {proposition}[theorem]    {Proposition}

\newcommand{\bb}{\mathbb}
\renewcommand{\rm}{\mathrm}
\newcommand{\wt}{\widetilde}
\newcommand{\cal}{\mathcal}

\theoremstyle{definition}

\newtheorem{remark}[theorem]{Remark}
\newtheorem{ex}[theorem]{Example}

\numberwithin{equation}{section}

\begin{document}

\title{Theta lifting for loop groups}

\date{\today}

\author[Dongwen Liu]{Dongwen Liu}

\address{School of Mathematical Sciences, Zhejiang University, Hangzhou 310027, Zhejiang, P.R. China}

\email{maliu@zju.edu.cn}

\author[Yongchang Zhu]{Yongchang Zhu$^{*}$}

\address{Department of Mathematics, The Hong Kong University of Science and Technology, Clear Water Bay, Kowloon, Hong Kong}

\thanks{$^{*}$ The second author's  research is supported by Hong Kong RGC grant  16305217 and 16305619.}

\email{mazhu@ust.hk}


\begin{abstract}
In this paper we study the theta lifting for loop groups and extend the classical tower property established by S. Rallis to the loop setting. As an application we obtain cusp forms on loop groups, and we give the first example where the cusp forms constructed using this method are nonvanishing.
\end{abstract}

\maketitle

\section{Introduction}

The study of automorphic forms on loop groups was initiated by Howard Garland.
 In a series of papers \cite{G3}--\cite{G8}, he has studied the Eisenstein series on loop groups  among many other
 things. Subsequent works in this area include studies of representations of loop groups over $p$-adic fields through various sophisticated techniques  \cite{BK1, BGKP, BKP, P2}, which will have important applications to automorphic forms.
   See also \cite{P1, Li1, GMP, LL} for  subsequent works  on loop Eisensteins series, and \cite{CLL1, CLL2, CGLLM} for the Kac-Moody Eisenstein series beyond the loop case. We refer the readers to \cite{BK2} for a more complete list of literatures.

On the other hand, in the joint work \cite{GZ1, GZ2} of H. Garland with the second named author, the Siegel-Weil formula was established for loop groups which relates the integral of theta functions and certain Eisenstein series on loop groups, in a way analogous to the finite dimensional situation. See \cite{Z1} for some background on  the theory of Weil representations and theta functionals for loop symplectic groups, and also  \cite{Li2, LZ, Z2} for more related results on these subjects.
The main goal of this paper is to  further study the theta lifting, and in particular adapt the so-called tower property to the loop group setting. To begin with, let us first explain the motivation and summarize some classical results in the finite dimensional case.

The method of theta lifting between reductive dual pairs via Weil representations provides an extremely useful tool and has wide applications in representation theory and automorphic forms. See \cite{H}
for a systematic formulation.  In the global theory, one considers the lifting of an irreducible automorphic cuspidal representation $\pi$ of a classical group $G$ to a Witt tower $\{G'_n\}_{n\geq 0}$ such that $(G,G'_n)$ form a family of reductive dual pairs. Here $n$ stands for the split rank, and the groups $G, G_n'$ sometimes are metaplectic covering groups. Let us denote the corresponding theta liftings by $\theta_{G,G'_n}(\pi)$. The remarkable tower property due to S. Rallis \cite{R1} asserts the following important fact. Let $P_{n,l}$ be the maximal parabolic subgroup of $G'_n$ which stabilizes an $l$-dimensional isotropic space, $l\leq n$, so that its Levi component is $\rm{GL}_l\times G'_{n-l}$. Then the constant term of $\theta_{G,G'_n}(\pi)$ along the unipotent radical of $P_{n,l}$, as a representation of  $G'_{n-l}$ is equal to $\theta_{G,G'_{n-l}}(\pi)$.

One can deduce several important consequences from the tower property, such as that the first occurrence of the theta liftings of $\pi$ is cuspidal while the latter occurrences are never cuspidal, and moreover the theta lifting has an involution property. These facts have been established for all reductive dual pairs in the works \cite{R1, R2, M1, M2, JS, Wu}. We should mention that in proving these facts one also needs the regularized Siegel-Weil formula which is another powerful tool in the study of automorphic forms.

Now we turn to the loop setting. We shall briefly give some details below which are necessary in order to formulate the main results and applications.
To simplify notations, for a finite dimensional vector space $\mathbb{W}$ over a field $F$ we denote
\[
\bb{W}((t))=\bb{W}\otimes_F F((t)),\quad \bb{W}[[t]]=\bb{W}\otimes_F F[[t]],\quad \bb{W}[t^{-1}]t^{-1}=\bb{W}\otimes_F F[t^{-1}]t^{-1}.
\]
Let $F$ be a number field and $\bb{A}$ be the adele ring of $F$. Similar notations apply for certain adelic spaces with $F((t))$ replaced by $\bb{A}\langle t\rangle$, which is a certain restricted product of
local Laurent series $F_v((t))$ (see Section 2.2). We also need $F\langle t\rangle:=F((t))\cap \bb{A}\langle t\rangle$. Let $(G,G')=(\rm{Sp}(W),\rm{O}(V))$ be a reductive dual pair over $F$ sitting inside $\rm{Sp}(\bb{W})$, where $W$ and $V$ are finite dimensional symplectic and orthogonal spaces over $F$ respectively, and $\bb{W}:=W\otimes_FV$.

Following \cite{Z1}, we have the adelic loop symplectic group $\widetilde{\mathrm{Sp}}(\bb{W}_{\bb A}\langle t\rangle)$, which is a central extension of $\mathrm{Sp}(\bb{W}_\bb{A}\langle t\rangle)$ and can be written as a certain restricted product of local loop groups. Its Weil representation, denoted by $\omega$, is realized on a certain space $\mathcal{S}'({\bb X}_{\bb A})$ of Bruhat-Schwartz functions, where ${\bb X}:={\bb W}[t^{-1}]t^{-1}$ is a Lagrangian space of $\bb{W}\langle t\rangle$. An important new feature in the loop setting is that in the definition of $\cal{S}'({\bb X}_{\bb A})$ one has to invoke a parameter of loop rotations $q: t\mapsto qt$, where $q\in\bb{A}^\times$ has idele norm $|q|>1$.  Then the theta functional for $\phi\in \mathcal{S}'({\bb X}_{\bb A})$,
\[
\theta_\phi(g):=\sum_{r\in {\bb X}_F}\omega(g)\phi(r)
\]
converges absolutely \cite{Z1,LZ}, and defines an automorphic function on the arithmetic quotient $\rm{Sp}(\bb{W}\langle t\rangle)\backslash \widetilde{\rm{Sp}}(\bb{W}_\bb{A}\langle t\rangle)$.

Let $\pi$ be an irreducible automorphic cuspidal representation of $G'(\bb{A})$. For $\phi$ as above and $f\in \pi$, we pull back $f$ to $G'(\bb{A}\langle t\rangle_+)$ via the evaluation map $G'(\bb{A}\langle t\rangle_+) \to G'(\bb{A})$ at $t=0$, where $\bb{A}\langle t\rangle_+=\bb{A}\langle t\rangle\cap \bb{A}[[t]]$, and define the theta lifting to $G(\bb{A}\langle t\rangle)$ by
\begin{equation}\label{1.1}
\theta^\phi_f(g)=\int_{G'(F\langle t\rangle_+)\backslash G'(\bb{A}\langle t\rangle_+)}\theta_\phi(g,h)f(h)dh,\quad g\in G(\bb{A}\langle t\rangle),
\end{equation}
noting that the quotient space $G'(F\langle t\rangle_+)\backslash G'(\bb{A}\langle t\rangle_+)$ is locally compact. We can show that (see Theorem \ref{weil} and Proposition \ref{prop3.2}) the theta function is of polynomial growth hence the theta integral converges thanks to the rapid decay of cusp forms. We also remark that in defining the theta lifting (\ref{1.1}) we take the integral only over $G'(F\langle t\rangle_+)\backslash G'(\bb{A}\langle t\rangle_+)$ instead of $G'(F\langle t\rangle)\backslash G'(\bb{A}\langle t\rangle)$, the main reason
 is that  $G'(\bb{A}\langle t\rangle_+)$ is the largest subgroup of $G'(\bb{A}\langle t\rangle)$
   that preserves $\bb{W}_{\bb A} [[ t]]$, see also the comments after the statement of loop version of Siegel-Weil formula in Section \S\ref{s3}.

Let $W=\ell^-\oplus W_0\oplus \ell^+$ be a decomposition where $W_0$ is nondegenerate and $W_0^\perp=\ell^-\oplus \ell^+$ is a polarization. Let $P$ be the maximal parabolic subgroup of  $\rm{Sp}(W\langle t\rangle)$ stabilizing $t^{-1}\ell^-\oplus W\langle t\rangle$. Then our main result Theorem \ref{main} states that the constant term of $\theta^\phi_f(g)$ along the unipotent radical of $P$ is equal to the classical theta lifting from $G'(\bb{A})$ to $\rm{Sp}(W_\ell)$, where $W_\ell:=t^{-1}\ell^-\oplus \ell^+$. As an immediate consequence, we obtain Corollary \ref{cor} that if the first occurrence of the classical theta liftings of $\pi$ takes place at $\widetilde{G}_n(\bb{A})=\widetilde{\rm{Sp}}_{2n}(\bb{A})$, then  its lifting to each $\widetilde{G}_j(\bb{A}\langle t\rangle)$, $j<n$, is cuspidal in the sense that the constant term along each standard proper parabolic subgroup is zero.

The organization as well as some other interesting features of this paper, are described as follows. In Section \S\ref{s2} we briefly recall the local and global theory of Weil representations and theta functionals for loop symplectic groups following \cite{Z1, LZ}. In Section \S\ref{s3} we establish the absolute convergence of theta integrals and extend the loop Siegel-Weil formula in \cite{GZ1, GZ2} beyond the anisotropic case. Namely, we prove a  loop analog of the Weil's convergence criterion for theta integrals (Theorem \ref{weil}) under which the Siegel-Weil formula (Theorem \ref{swf}) holds. The proof invokes A. Weil's original arguments \cite{W2}, as well as an inequality on Arakelov divisors \cite{GS} which was used already in our previous work \cite{LZ}. In Section \S\ref{s4} we prove our main results Theorem \ref{main} and Corollary \ref{cor} under the assumption that $G'=\rm{O}(V)$ is anisotropic. In this case one does not need the cuspidality and in fact one only needs to take the theta integral over $G'(F)\backslash G'(\bb{A})$ in (\ref{1.1}). Then we also show that such modified theta liftings are always nonzero (Proposition \ref{propf}) as long as the parameter $q$ has large idele norm. Applying this result, we construct two explicit examples of nonzero cusp forms on the loop group of $\rm{SL}_2$, based on the work of R.Howe and I.I. Piatetski-Shapiro \cite{H-PS} and H. Yoshida \cite{Y} respectively. To our best knowledge these are the first examples of nonzero cusp forms on loop groups. Finally in Section \S\ref{s5} we prove Theorem \ref{main} in general. The proof is quite involved and makes use of some variants of several ingredients from \cite{GZ1, GZ2}. The rough idea is to analyze the affine Gra\ss mannian and loop group orbits, which enable us to break the theta integral into a bunch of orbital integrals. Then we show that the nonvanishing orbital integrals are exactly those we need in the tower property.

Let us further comment on the representation-theoretic aspects of our main result. 
The existence of cusp forms for finite dimensional groups is best considered in the
context of  trace formula, while the spectral theory and trace formula for automorphic forms on loop
groups seem out of reach for now. The construction of cuspidal functions on loop groups in this paper
is analogous to the earlier explicit construction of cuspidal Maass forms using real quadratic fields. The cuspidal functions produced here are in fact associated with  loop group representations. We expect that these functions should be finite under the action of the central elements in the completed universal enveloping algebra of the affine Kac-Moody algebra at archimedean places, and should be Hecke eigenforms at unramified non-archimedean places.  As  some first evidence, we mention that

\begin{itemize}
\item  for an archimedean
component of the Weil representation of the adelic loop symplectic group in this paper,  we can construct the corresponding representation of the affine Kac-Moody algebra, which is in analogy with the free field realization of the algebra; 

\item for a $p$-adic place with $p\neq 2$,  our Weil representation contains a $K$-fixed function.  The
spherical Hecke algebra studied in \cite{BK1, BKP}  (more precisely an extension of it to
metaplectic loop groups) should act on the $K$-fixed function as an eigenfunction.
\end{itemize}

Finally, in a subsequent work  we will follow the arguments in \cite{G-PS-R, L, PS-R, R1} to  further develop a loop analog of the Rallis inner product formula.  The main ingredients we need in proving such a formula include some geometry of affine Gra\ss mannians, a new extension of the loop Siegel-Weil formula as well as the Satake isomorphism for loop groups \cite{BKP}.

{\bf Acknowledgement.} The authors would like to thank Howard Garland, Dihua Jiang, Jianshu Li, Manish Patnaik, Chenyan Wu and Jun Yu for some helpful discussions during the preparation of this work. The authors thank the anonymous referees for their valuable comments and suggestions which significantly improve the paper. 

\section{Weil representations and theta functionals of symplectic loop groups}\label{s2}

In this section we briefly recall the Weil representations of symplectic loop groups over local and global fields following \cite{Z1} and \cite[Section 2]{GZ2}, as well as the absolute convergence of theta functionals \cite[Theorem 3.1]{LZ}.

\subsection{Local theory} \label{2.1} Let $F$ be a field of characteristic 0, and $W$ be a symplectic vector space of dimension $2n$ over $F$, with symplectic form denoted by $\langle,\rangle_F$. It gives rise to an $F((t))$-valued symplectic form $\langle,\rangle_{F((t))}$ on $W((t))$ by scalar extensions. It further gives
an $F$-valued symplectic form $\langle,\rangle$ on $W((t))$ by taking the residue
\begin{equation}\label{form}
\langle w, w'\rangle=\textrm{Res }\langle w, w'\rangle_{F((t))},
\end{equation}
where $w, w'\in W((t))$, and Res $a$ for $a\in F((t))$ is the coefficient of $t^{-1}$ in $a$. It should be pointed out that, with this definition $\langle, \rangle_F$ is {\it not} the restriction of
$\langle,\rangle$ to the subspace $W$. Moreover for an $F$-algebra $A$ we write $\langle,\rangle_A$ for the symplectic form on $W_A:=W\otimes_FA$ given by scalar extension from $\langle,\rangle_F$. For example, later  we will occasionally write $\langle, \rangle_\mathbb{A}$ for a symplectic form on $W_\bb{A}$, where $\mathbb{A}$ is the adele ring of a number field $F$. Getting back to the general discussion, the spaces
\begin{equation}\label{lag}
X=W[t^{-1}]t^{-1}, \quad Y=W[[t]]
\end{equation}
are maximal isotropic subspaces of $W((t))$, also  called Lagrangian subspaces. We also write for example $X=X_F$ if we want to specify the field $F$. Denote by
\begin{equation}\label{pro}
p_X: W((t))\to X, \quad p_Y: W((t)) \to Y
\end{equation}
the natural projections. Let $\textrm{Sp}(W((t)),Y)$ be the group of all $F$-linear symplectic isomorphisms $g$ of $W((t))$ such that $Yg$ and $Y$ are commensurable, which contains $\textrm{Sp}_{2n}(F((t)))$ as a subgroup. Here by convention $\textrm{Sp}(W((t)),Y)$ acts on $W((t))$ from the right. Then with respect to the decomposition $W((t))=X\oplus Y$, each $g\in\textrm{Sp}(W((t)),X)$ can be represented by a matrix
\begin{equation} \label{matrix}
g=\begin{pmatrix} a_g & b_g \\  c_g & d_g\end{pmatrix}.
\end{equation}
Note that the image of $c_g: Y\to X$ is a finite dimensional space over $F$.
The group law of the Heisenberg group $H=W((t))\times F$ is defined by
\[
(w, z)(w',z')=(w+w', \frac{1}{2}\langle w, w'\rangle+z+z').
\]
Then $\textrm{Sp}(W((t)),Y)$ acts on $H$ from the right by $(w,z)\cdot g=(w\cdot g, z)$.

Now assume further that $F$ is a local field.
Let $\mathcal{S}(X)$ be the space of complex valued Schwartz functions on $X$, i.e. functions whose restriction to each finite dimensional subspace is a Schwartz function in the usual sense.
Typical examples include

\begin{itemize}

\item the characteristic function of $t^{-1}\mathcal{O}[t^{-1}]^{2n}\subset X=t^{-1}F[t^{-1}]^{2n}$, where $F$ is a $p$-adic local field with ring of integers $\mathcal{O}$ and  $W=F^{2n}$ is a symplectic space;

\item  the Gaussian
function $e^{\pi iQ(x)}$, where $F$ is archimedean and $Q(x)$ is a complex valued quadratic form on  $X$ with positive definite imaginary part. An example for $F=\mathbb{R}$ is given by $Q(x)=(x, x\Omega)$ for $\Omega=iI$ the purely imaginary element in the loop upper half-space defined in \cite[Section 4]{Z1}.
\end{itemize}
Fix a non-trivial additive character $\psi$ of $F$. Then the Heisenberg group $H$ acts on $\mathcal{S}(X)$ in the usual way such that the central element $(0,z)$ acts by the scalar
$\psi(z)$.
For $g\in \textrm{Sp}(W((t)), Y)$ with decomposition (\ref{matrix}) and a choice of Haar measure on $\textrm{Im }c_g$, define an operator $T_g$ on $\mathcal{S}(X)$ by
\begin{equation}\label{action}
(T_g\phi)(x)=\int_{\textrm{Im }c_g}S_g(x+y)\phi(xa_g+ yc_g)d(yc_g),
\end{equation}
where $\phi\in \mathcal{S}(X)$ and
\[
S_g(x+y)=\psi\left(\frac{1}{2}\langle xa_g, xb_g\rangle+\frac{1}{2}\langle yc_g, yd_g\rangle+\langle yc_g, xb_g\rangle\right).
\]
In particular, if $c_g=0$ then
\begin{equation}\label{gamma0}
(T_g\phi)(x)=\psi\left(\frac{1}{2}\langle xa_g, xb_g\rangle\right)\phi(xa_g).
\end{equation}
Note that more explicitly one has
\[
xa_g= p_X(xg), \quad xb_g= p_Y(xg), \quad xg=xa_g+xb_g,
\]
and it is clear that
\begin{equation}\label{rel}
\langle x_1a_g, x_2 b_g\rangle=\langle x_1g, x_2b_g\rangle=\langle x_1a_g, x_2g\rangle, \quad x_1, x_2\in X.
\end{equation}
The operator $T_g$ is compatible with the Heisenberg group action, i.e. for $h\in H$ one has
\[
T_g^{-1}h T_g=h\cdot g.
\]
It was proved in \cite[Proposition 2.8]{Z1} that $T_{g_1}T_{g_2}$ coincides with $T_{g_1g_2}$ up to a scalar, and $g\mapsto T_g$ gives a projective representation of $\textrm{Sp}(W((t)),Y)$ on $\mathcal{S}(X)$.  By restriction, we obtain a projective representation of $\textrm{Sp}_{2n}(F((t)))$ on $\mathcal{S}(X)$. Then the cocycle of this representation is given by 
\cite[Theorem 3.3]{Z1}: for $a, b\in F((t))^\times$,
\begin{equation} \label{symbol}
(a,b)=\frac{\gamma(a_0,\psi)^{\varepsilon(a)}\gamma(b_0,\psi)^{\varepsilon(b)}}{\gamma(a_0b_0,\psi)^{\varepsilon(ab)}}\left|C(a,b)\right|^{-1/2},
\end{equation}
where  
\begin{itemize}
\item $a=a_0 t^{v(a)}u_a$, $b=b_0 t^{v(b)}u_b$ with $a_0, b_0\in F^\times$, $u_a, u_b\in 1+tF[[t]]$, \\
\item $\gamma(c_0,\psi)$, $c_0\in F^\times$, is the Weil index of the distribution $\psi(\frac{1}{2}c_0x^2)$ (see e.g. \cite{W1}), \\
\item $\varepsilon(c)=0$ or $1$ according to the valuation $v(c)$ of $c\in F((t))^\times$ with respect to the local parameter $t$ is even or odd, \\
\item $C(a,b)$ is the usual tame symbol, and $|\cdot|$ is the absolute value on $F$ associated to a Haar measure.
\end{itemize} 
We define the metaplectic loop group $\widetilde{\textrm{Sp}}_{2n}(F((t)))$ to be the central extension for the symbol (\ref{symbol})
\begin{equation}\label{cen}
1\longrightarrow \mathbb{C}^\times \longrightarrow \widetilde{\textrm{Sp}}_{2n}(F((t))) \longrightarrow \textrm{Sp}_{2n}(F((t)))\longrightarrow 1,
\end{equation}
and call its representation on $\mathcal{S}(X)$ as above the Weil representation, denoted by $\left(\omega, \mathcal{S}(X)\right)$.

We also need to introduce the reparametrization group
\[
\textrm{Aut }F((t))=\left\{\sum^\infty_{i=1}a_it^i\in F[[t]]t: a_1\neq 0\right\},
\]
which plays an important role in the loop setting as we will see shortly. The group law is given by $(\sigma_1 \cdot \sigma_2)(t)=\sigma_2(\sigma_1(t))$, $\sigma_1, \sigma_2\in \textrm{Aut }F((t))$. It acts on $F((t))$ and the formal 1-forms $F((t))dt$ from the right by
\begin{equation} \label{rot-act}
a(t)\cdot\sigma(t)=a(\sigma^{-1}(t))\textrm{\quad and \quad} a(t)dt\cdot \sigma(t)=a(\sigma^{-1}(t))d\sigma^{-1}(t)
\end{equation}
respectively. Here for an element $\sigma(t)=\sum^\infty_{i=1}a_it^i \in \textrm{Aut }F((t))$, we denote $d\sigma(t)$ the 1-form 
\[
d\sigma(t)=\sum^\infty_{i=1} i a_i t^{i-1}dt.
\]
Note that $\textrm{Aut }F((t))$ contains the subgroup $\{ct: c\in F^\times\}$, the ``rotation of loops". Assume that the symplectic form $\langle,\rangle$ on $F^{2n}$ is represented by the matrix $\begin{pmatrix} 0 & I \\ -I & 0\end{pmatrix}$. View the first $n$ components of $F((t))^{2n}$ as elements in $F((t))$, and the last $n$ components as 1-forms in $F((t))dt$ (without writing $dt$), and let $\textrm{Aut }F((t))$ act them correspondingly. Then the action of $\textrm{Aut }F((t))$ on $W((t))$ preserves the symplectic form, so that $\textrm{Aut }F((t))$ can be viewed as a subgroup of $\textrm{Sp}(W((t)),Y)$. Thus it is clear that elements of $\textrm{Aut }F((t))$ act on $\mathcal{S}(X)$ by the formula (\ref{gamma0}). The group $\textrm{Aut }F((t))$ also acts on $\textrm{Sp}_{2n}(F((t)))$ as automorphisms, which lifts to an action on
$\widetilde{\textrm{Sp}}_{2n}(F((t)))$ and one may form the semi-direct product $\widetilde{\textrm{Sp}}_{2n}(F((t)))\rtimes \textrm{Aut }F((t))$. The readers are referred to \cite[Section 2.1]{GZ2} for more explanations about the action 
of the reparametrization group. 

It is well known that the loop group $\widetilde{\textrm{Sp}}_{2n}(F((t)))$ can be interpreted as an infinite dimensional Kac-Moody group corresponding to the untwisted affine Lie algebra associated to $\frak{sp}_{2n}$, whose Dynkin diagram is given by
\begin{equation}\label{dynkin}
\xymatrix{
\alpha_0\ar@{=>}[r]  &  \alpha_1\ar@{-}[r] & \alpha_2\ar@{-}[r] &  \cdots  \cdots \ar@{-}[r] & \alpha_{n-1}\ar@{<=}[r]& \alpha_n}
\end{equation}
Thus it has a presentation using Chevalley generators and relations, together with certain process of taking completions. We refer the readers to \cite[Section 2.1]{GZ2} for more details. Let $B_o$ be the standard Borel subgroup of $\textrm{Sp}_{2n}(F)$, and $B'$ be the preimage of $B_o$ under the projection $\textrm{Sp}_{2n}(F[[t]])\to \textrm{Sp}_{2n}(F)$. By 
\cite[Lemma 2.4]{GZ2}, the central extension (\ref{cen}) splits over $\textrm{Sp}_{2n}(F[[t]])$, and therefore the preimage of $B_o$ in $\widetilde{\textrm{Sp}}_{2n}(F((t)))$ is $B:=B'\times\mathbb{C}^\times$, which is called the standard Borel subgroup of $\widetilde{\textrm{Sp}}_{2n}(F((t)))$. The theory of Tits system shows that there is a Bruhat decomposition
\begin{equation}\label{bruhat}
\widetilde{\textrm{Sp}}_{2n}(F((t)))=B\widetilde{\mathcal{W}}B,
\end{equation}
where $\widetilde{\mathcal{W}}$ is the affine Weyl group. It is known that $\widetilde{\mathcal{W}}$ is isomorphic to the semi-direct product $\mathcal{W}\ltimes Q^\vee$, where $\mathcal{W}$ is the finite Weyl group and
$Q^\vee$ is the coroot lattice of $\frak{sp}_{2n}$.

We can define a ``maximal compact subgroup" $K$ of $\widetilde{\textrm{Sp}}_{2n}(F((t)))$ so that its intersection with the center $\mathbb{C}^\times$ is $S^1$, and its image $K'$ in $\textrm{Sp}_{2n}(F((t)))$ is as follows. If $F$ is a $p$-adic field with ring of integers $\mathcal{O}$, then $K'=\textrm{Sp}_{2n}(\mathcal{O}((t)))$; if $F$ is archimedean, then
\[
K'=\{g\in \textrm{Sp}_{2n}(F[t,t^{-1}]): \overline{g(t)}g(t^{-1})^T=I_{2n}\},
\]
where $\overline{(\cdot)}$ stands for the complex conjugate, and $(\cdot)^T$ is the matrix transpose. The standard use of BN-pairs as in \cite{S} shows that it holds the Iwasawa decomposition (cf. \cite[Theorem 16.8]{G1} 
and \cite[(2.15)]{GZ2})

\begin{equation}\label{iwa}
\widetilde{\textrm{Sp}}_{2n}(F((t)))=BK.
\end{equation}

We recall the following result given by \cite[Lemma 2.5 and 2.6]{GZ2}.

\begin{lemma}\label{fun} (i) If $F$ is non-archimedean with odd residue characteristic and ring of integers $\mathcal{O}$, and if the conductor of $\psi$ is $\mathcal{O}$, then the characteristic function
$\phi_0$ of $t^{-1}\mathcal{O}[t^{-1}]^{2n}$ is fixed by $K$.

(ii) If $F$ is archimedean, then there is a nonzero element $\phi_0\in \mathcal{S}(X)$ fixed by $K$ up to a scalar.

\end{lemma}

We mention that for $F=\mathbb{R}$ and $\psi(x)=e^{\pm 2\pi i x}$, the element $\phi_0$ can be taken to be the standard Gaussian function $e^{\pi i (x, x\Omega)}$ with $\Omega=i I$ that we introduced earlier. See the proof of 
\cite[Lemma 2.5]{GZ2} for the complex analog.

\subsection{Global theory} From now on we assume that $F$ is a number field. Thus for each place $v$ of $F$ one has the local Weil representation $\omega_v$ of $\widetilde{\textrm{Sp}}_{2n}(F_v((t)))\rtimes
\textrm{Aut }F_v((t))$ on $\mathcal{S}(X_{v})$, where
\[X_v=W_v[t^{-1}]t^{-1},\quad W_v=W\otimes_F F_v.\]
Let $\mathbb{A}$ be the adele ring of $F$, and $\psi=\bigotimes\limits_v\psi_v$ be a non-trivial character of $\mathbb{A}/F$. For a finite place $v$, let $\mathcal{O}_v$ be the ring of integers of $F_v$. Define
\begin{equation} \label{A<t>}
\mathbb{A}\langle t\rangle =\left\{(a_v)\in\prod_v F_v((t)): a_v\in \mathcal{O}_v((t))\textrm{ for almost all finite places }v\right\},
\end{equation}
and 
\begin{equation} \label{F<t>}
F\langle t\rangle= F((t))\cap\mathbb{A}\langle t\rangle,
\end{equation} 
which is a subfield of $F((t))$. We have a short exact sequence 
\begin{equation}
1\longrightarrow \bigoplus_v \mathbb{C}^\times\longrightarrow \prod'_v \widetilde{\textrm{Sp}}_{2n}(F_v((t))\longrightarrow \textrm{Sp}_{2n}(\mathbb{A}\langle t\rangle)\longrightarrow 1,
\end{equation}
where $\prod'_v \widetilde{\textrm{Sp}}_{2n}(F_v((t))$ is the restricted product with respect to the ``maximal compact" subgroups $K_v$. The adelic metaplectic loop group $\widetilde{\textrm{Sp}}_{2n}(\mathbb{A}\langle t\rangle)$ with
$\mathbb{A}\langle t\rangle$ given by \eqref{A<t>}, is defined by pushing out the above exact sequence along the product map $\bigoplus_v \mathbb{C}^\times \to \mathbb{C}^\times$, that is, we have a commutative diagram with exact rows
\begin{equation}\label{ade}
\xymatrix{
1 \ar[r] & \bigoplus_v\mathbb{C}^\times \ar[r] \ar[d]  & \prod'_v \widetilde{\textrm{Sp}}_{2n}(F_v((t)) \ar[r] \ar[d] &  \textrm{Sp}_{2n}(\mathbb{A}\langle t\rangle) \ar[r] \ar@{=}[d] & 1\\
1  \ar[r] &  \mathbb{C}^\times  \ar[r] & \widetilde{\textrm{Sp}}_{2n}(\mathbb{A}\langle t\rangle) \ar[r] &
\textrm{Sp}_{2n}(\mathbb{A}\langle t\rangle) \ar[r] & 1
}
\end{equation}
The adelic group $\textrm{Aut }\mathbb{A}\langle t\rangle$ can be defined similarly; we refer the readers to \cite[Section 2]{LZ} for the precise definition. By the proof of \cite[Lemma 2.9]{GZ2}, the local symbols (\ref{symbol}) satisfy the product formula, hence the central extension (\ref{ade}) splits over $\textrm{Sp}_{2n}(F\langle t\rangle)$.

By Lemma \ref{fun} for almost all finite places $v$ there is $\phi_{0,v}\in \mathcal{S}(X_{v})$ fixed by $K_v$. The restricted tensor product $\mathcal{S}(X_{\mathbb{A}}):=\bigotimes'_v \mathcal{S}(X_{v})$ with respect to $\{\phi_{0,v}\}$, that is, spanned by functions of the form $\bigotimes_v\phi_v$, where $\phi_v\in \mathcal{S}(X_v)$ and $\phi_v=\phi_{0, v}$ for almost all finite spaces $v$, is a representation of $\widetilde{\textrm{Sp}}_{2n}(\mathbb{A}\langle t\rangle)$, called the global Weil representation. Let us denote this representation by $\omega=\bigotimes'_v\omega_v$.
We need to further introduce a suitable subrepresentation, on which the theta functional
\begin{equation}\label{theta}
\phi\mapsto \theta(\phi)=\sum_{r\in X_{F}}\phi(r)
\end{equation}
defined in the usual way, is absolutely convergent.

For a finite place $v$, a subgroup of the Heisenberg group $H_v=F_v((t))^{2n}\times F_v$ is called a congruence subgroup if it contains $\varpi_v^k((t))^{2n}$ for some integer $k$, where $\varpi_v\in\mathcal{O}_v$ is a local parameter. A function $\phi_v\in \mathcal{S}(X_{v})$ is called elementary if it is bounded and fixed by a congruence subgroup of
$H_v$. For example the function $\phi_0$ in Lemma \ref{fun} (i) is elementary. Let $\mathcal{E}(X_{\mathbb{A}})$ be the space of functions on $X_{\mathbb{A}}$ which are finite linear combinations of $\omega(g)\cdot\bigotimes\limits_v \phi_v$, where
$g\in\widetilde{\textrm{Sp}}_{2n}(\mathbb{A}\langle t\rangle)$, and the following hold:

\begin{itemize}

\item $\phi_v=\phi_{0,v}$ for almost all finite places $v$;

\item $\phi_v$ is elementary for each remaining finite place $v$;

\item $\phi_v=P_v\cdot \phi_{0,v}$ for each archimedean place $v$, where $P_v$ is a polynomial function on some finite dimensional subspace $W_v[t^{-1},\ldots, t^{-N}]$ of $X_v$.

\end{itemize}
It is clear that $\mathcal{E}(X_{\mathbb{A}})$ is a subrepresentation of $\mathcal{S}(X_{\mathbb{A}})$.

We also have to introduce a sub-semigroup of $\textrm{Aut }\mathbb{A}\langle t\rangle$ by
\[
\textrm{Aut}^+\mathbb{A}\langle t\rangle=\left\{\left(\sum^\infty_{i=1}a_{i,v}t^i\right)_v\in\textrm{Aut }\mathbb{A}\langle t\rangle: \prod_v |a_{1,v}|>1\right\}.
\]
Then we let $\mathcal{S}'(X_{\mathbb{A}})=\mathrm{Aut}^+\mathbb{A}\langle t\rangle  \cdot \mathcal{E}(X_{\mathbb{A}})$,
where $\mathrm{Aut}^+\mathbb{A}\langle t\rangle$ acts on $\mathcal{E}(X_{\mathbb{A}})$ from the left by change of variables (cf. \eqref{rot-act})
\[
\left(\sigma(t)f\right)(a(t))=f(a(t)\cdot \sigma^{-1}(t))=f(a(\sigma(t)), 
\]
for $\sigma(t)\in \textrm{Aut}^+\mathbb{A}\langle t\rangle$, $f\in  \mathcal{E}(X_{\mathbb{A}})$ and $a(t)\in X_{\mathbb{A}}$. 
 For later use, we mention that
\[
\textrm{Aut}^+\mathbb{A}\langle t\rangle = \mathbb{A}^\times_{>1}t\ltimes \textrm{Aut}^0\mathbb{A}\langle t\rangle,
\]
where  
\begin{align*}
& \mathbb{A}^\times_{>1}=\{q=(q_v)_v\in \mathbb{A}^\times: |q|=\prod_v|q_v|>1\}, \\
&
\mathbb{A}^\times_{>1}t=\left\{qt\in\textrm{Aut }\mathbb{A}\langle t\rangle: q\in \mathbb{A}^\times_{>1}\right\},\\
&
\textrm{Aut}^0\mathbb{A}\langle t\rangle=\left\{\left(t+\sum^\infty_{i=2}a_{i,v}t^i\right)_v\in\textrm{Aut }\mathbb{A}\langle t\rangle\right\}.
\end{align*}
In particular we have a disjoint union over the loop rotation parameter $q$,
\begin{equation}\label{part}
\mathcal{S}'(X_{\mathbb{A}})=\bigsqcup_{q\in \mathbb{A}^\times_{>1}}\mathcal{S}'(X_{\mathbb{A}})_q,
\end{equation}
where
\[
\mathcal{S}'(X_{\mathbb{A}})_q=qt\cdot \textrm{Aut}^0\mathbb{A}\langle t\rangle\cdot \mathcal{E}(X_{\mathbb{A}}).
\]

The main convergence result is given by \cite[Theorem 3.1]{LZ}:

\begin{theorem}\label{conv}
If $\phi \in  \mathcal{S}'(X_{\mathbb{A}})$, then the theta series $\theta(\phi)$ converges absolutely and is invariant under $\mathrm{Sp}_{2n}(F\langle t\rangle)$.
\end{theorem}

\section{Absolute convergence of theta integrals and Siegel-Weil formula} \label{s3}

In this section we study the convergence of theta integrals which are used to define the theta lifting from an orthogonal loop group to a symplectic loop group.
We continue to assume that $F$ is a number field. %To simplify notations, in this paper for an algebraic group $H$ over the number field $F$ (possibly infinite dimensional, e.g. $H$ being pro-unipotent), we denote by $[H]$ for short the arithmetic quotient space $H(F)\backslash H(\mathbb{A})$.

Let $W$ be a symplectic space of dimension $2n$ over $F$ and $G=\textrm{Sp}(W)$. Let $V$ be a vector space over $F$ of dimension $m$ with a non-degenerate symmetric bilinear form $(,)$, and let $G'=\textrm{O}(V)$ be the isometry group. Then
\[
\bb{W}=W\otimes_F V
\]
is a symplectic space of dimension $2N:=2mn$, and $(G, G')$ is a reductive dual pair (see \cite{H}) inside $\textrm{Sp}(\bb{W})=\rm{Sp}_{2N}$. Similar to (\ref{lag}) in the last section, let
\[
\bb{X}=\bb{W}[t^{-1}]t^{-1}=X\otimes_F V,\quad \bb{Y}=\bb{W}[[t]]=Y\otimes_F V
\] which form a polarization of $\bb{W}((t))$, so that $\widetilde{\textrm{Sp}}_{2N}(\mathbb{A}\langle t\rangle)$ acts on $\mathcal{S}(\bb{X}_{\mathbb{A}})$ by the Weil representation. According to \cite{GZ2}, for the purpose of theta lifting it suffices to consider the parabolic subgroup $G'(F\langle t\rangle_+)$ and its adelic group $G'(\mathbb{A}\langle t\rangle_+)$, where $F\langle t\rangle_+=F\langle t\rangle\cap F[[t]]$ and $\mathbb{A}\langle t\rangle_+=\mathbb{A}\langle t\rangle\cap \mathbb{A}[[t]]$. Recall that $F\langle t\rangle$ is the subfield of $F((t))$ defined by \eqref{F<t>}. Note that
$F\langle t\rangle_+\backslash \mathbb{A}\langle t\rangle_+\cong F[[t]]\backslash \mathbb{A}[[t]].$

By \cite[Lemma 2.4]{GZ2} each local metaplectic cover splits over $G'(F_v[[t]])\subset \textrm{Sp}_{2N}(F_v[[t]])$, hence we may regard $G'(F_v[[t]])$ as a subgroup of $\widetilde{\textrm{Sp}}_{2N}(F_v((t)))$, and we have a dual pair
\[
(\wt{G}(F_v((t))), G'(F_v[[t]]))\hookrightarrow \widetilde{\textrm{Sp}}_{2N}(F_v((t))),
\]
where a central element $z\in \mathbb{C}^\times$ of $\widetilde{G}(F_v((t)))$ is mapped to the central element $z^m$ of $\widetilde{\textrm{Sp}}_{2N}(F_v((t)))$ under the above embedding (see \cite[Section 4]{GZ2}). Similarly one has an adelic dual pair
\[
(\widetilde{G}(\mathbb{A}\langle t\rangle), G'(\mathbb{A}\langle t\rangle_+)\hookrightarrow \widetilde{\textrm{Sp}}_{2N}(\mathbb{A}\langle t\rangle).
\]
Using (\ref{action}) one can show that the actions of $\widetilde{G}(\mathbb{A}\langle t\rangle)$ and $G'(\mathbb{A}\langle t\rangle_+)$ on $\mathcal{S}(\bb{X}_{\mathbb{A}})$ commute.

Let $\pi$ be an irreducible cuspidal automorphic representation of $G'(\mathbb{A})$, and $f\in\pi$ be a cusp form on $G'(\bb{A})$.
By abuse of notation, we also denote by $f$ its lift to $G'(\bb{A}\langle t\rangle_+)$,
\[
f: G'(\bb{A}\langle t\rangle_+)\to\mathbb{C},\quad h\mapsto f(h(0)).
\]
Then it is clear that $f$ is bi-invariant under the unipotent radical
\begin{equation}\label{ug}
N(\mathbb{A})=\{h\in G'(\mathbb{A}\langle t\rangle_+): h(0)=1\}.
\end{equation}
 In the definition of theta lifting below as well as latter parts of the paper, we will frequently make use of the following identification without further comments
\[
G'(F[[t]])\backslash G'(\mathbb{A}[[t]]) \cong G'(F\langle t\rangle_+)\backslash G'(\mathbb{A}\langle t\rangle_+),
\]
thanks to that $F\langle t\rangle_+\backslash \mathbb{A}\langle t\rangle_+\cong F[[t]]\backslash \mathbb{A}[[t]]$
and that $N(\mathbb{A})$ is pro-unipotent.

Let $\phi\in \mathcal{S}'(\bb{X}_{\mathbb{A}})$ and consider the theta liftings to $g\in \widetilde{G}(\mathbb{A}\langle t\rangle)$,
\begin{align}
\label{theta1}\theta_\phi(g)& = \int_{G'(F[[t]])\backslash G'(\mathbb{A}[[t]])} \theta_\phi(g, h) dh,\\
\label{thetaf}\theta^f_\phi(g)&=\int_{G'(F[[t]])\backslash G'(\mathbb{A}[[t]])} \theta_\phi(g, h)f(h)dh,
\end{align}
where
\begin{equation}\label{thetagh}
\theta_\phi(g,h):=\theta\left(\omega(g,h)\phi\right)=\sum_{r\in \bb{X}_{F}}\omega(g,h)\phi(r),
\end{equation}
and $dh$ is a product of a Haar measure on $G'(F)\backslash G'(\bb{A})$ and a probability measure on the pro-unipotent quotient space $N(F)\backslash N(\bb{A})$. Note that $N(F)\backslash N(\bb{A})$ is topologically a countable direct product of copies of the compact space $F\backslash \mathbb{A}$, which is thereby compact and inherits a product measure. It is apparent that  (\ref{thetaf}) is bounded by a multiple of (\ref{theta1}), and both theta integrals converge absolutely if $V$ is anisotropic (see \cite{GZ2}), in which case
$G'(F[[t]])\backslash G'(\mathbb{A}[[t]])$ is compact. The aim of this section is to prove the convergence results of these theta integrals in general, so that the theta lifting makes sense and gives automorphic functions on the  loop group arithmetic quotient $G(F\langle t\rangle)\backslash \widetilde{G}(\mathbb{A}\langle t\rangle)$.

We will first show that (\ref{theta1}) converges under a loop analog of Weil's criterion, and as a byproduct extend the Siegel-Weil formula established in \cite{GZ1, GZ2} to general orthogonal groups, which are not necessarily anisotropic. Recall that we have the partition (\ref{part}).

\begin{theorem} \label{weil}
There is a constant $c>0$ such that if $V$ is anisotropic or $m-b>cn/\log |q|+1$, where $b$ is the Witt index of $V$, then the theta integral (\ref{theta1})  is absolutely convergent for any $\phi\in \mathcal{S}'(\bb{X}_{\mathbb{A}})_q$. In particular for $|q|$ large enough, if $V$ is anisotropic or $m\geq 3$, then $(\ref{theta1})$ is absolutely convergent for any $\phi\in \mathcal{S}'(\bb{X}_{\mathbb{A}})_q$.
\end{theorem}

\begin{proof}
We give a sketch of the proof, applying the original arguments in \cite{W2} as well as some new ingredients from \cite{LZ} for the loop setting. Since $N(F)\backslash N(\bb{A})$ is compact and $G'(\mathbb{A}\langle t\rangle_+)=G'(\mathbb{A})\ltimes N(\mathbb{A})$,   we can reduce the integral \eqref{theta1} to the integral over $G'(F)\backslash G'(\mathbb{A})$ and  it suffices to estimate
\[
\int_{G'(F)\backslash G'(\bb{A})}\theta_\phi(g,h)dh.
\]
Let $P$ be a minimal parabolic subgroup of $G'$, and $T\subset P$ be a maximal split torus so that $T\cong \mathbb{G}_m^b$. Let $\Theta(\mathbb{G}_m)\cong \mathbb{R}_+^\times$ be the subgroup of $\mathbb{A}^\times$ which consists of ideles $a_\tau=(a_v)_v$ such that $a_v=\tau\in\mathbb{R}_+^\times$ at infinite places $v$ and $a_v=1$ at finite places. Denote by $\Theta(T)$ the image of $\Theta(\mathbb{G}_m)^b$ under the isomorphism $T\cong\mathbb{G}_m^b$. Let $T(\mathbb{A})^+$ be the set of $a\in T(\mathbb{A})$ such that $|\alpha(a)|_\mathbb{A}\geq 1$ for any positive root $\alpha$, and $\Theta^+=\Theta(T)\cap T(\mathbb{A})^+$. The reduction theory as in \cite{W2} shows that it suffices to estimate the integral
\begin{equation}\label{est0}
\int_{\Theta^+}\theta_\phi(g, a)\cdot |\Delta_P(a)|_\mathbb{A}da,
\end{equation}
where $\Delta_P$ is the modular character of $P$. Then the problem is reduced to the estimation of
\[
\theta_\phi(g,a)=\sum_{r\in \bb{X}_{F}}\omega(g,a)\phi(r), \quad a\in\Theta^+.
\]
Applying the Fourier-type arguments in \cite[Lemma 4.8]{Z1} and \cite{LZ}, we may drop the $\textrm{Aut}^0\mathbb{A}\langle t\rangle$-component from $\phi$ and further assume that
$\phi=(q_vt)_v\cdot\bigotimes\limits_v\phi_v\in \mathbb{A}^\times_{>1}t\cdot \mathcal{E}(\bb{X}_{\mathbb{A}})$ is a pure tensor.
Let $S=S_\rm{fin}\cup S_\infty$ be a finite set of places including all the infinite places and the finite places $v$ with $\phi_v\neq\phi_{v,0}$ or $|q_v|_v\neq 1$. We may assume that each $\phi_v$, $v\in S_\textrm{fin}$ is bounded by a multiple of the characteristic function $\phi_v'$ of $\varpi_v^{k_v}\cal{O}_v^{2N}[t^{-1}]t^{-1}$ for some integer $k_v$; each $\phi_v$, $v\in S_\infty$ is a translate of $\phi_{0,v}$, hence is bounded by a multiple of a fixed Gaussian function, which implies that
\[
\left((q_vt)\cdot\phi_v\right)(r)\leq C_\phi\cdot\exp\left(-C\sum^{2N}_{i=1}\sum^\infty_{j=1}|r_{ij}|_v^2 |q_v|_v^{2j}\right)
\]
for some positive constants $C$ (not depending on $\phi$) and $C_\phi$, where
\[
r=\left(\sum^\infty_{j=1}r_{ij}t^{-j}\right)_{i=1,\ldots, 2N}\in \bb{X}_{v}\cong t^{-1}F_v[t^{-1}]^{2N}.
\]
Define $\mathcal{O}_S=F\cap \left(\prod_{v\not\in S}\mathcal{O}_v\right)$. Let $\lambda$'s be the characters of $T$ by which it acts on the space $V$, each with multiplicity $m_\lambda$, so that $\sum_\lambda m_\lambda=\dim V=m$. Moreover for $a\in \Theta(T)$ and a character $\lambda$ of $T$, we write $\lambda(a)\in \mathbb{R}^\times_+$ via the identification  $\Theta(\mathbb{G}_m)\cong \mathbb{R}_+^\times$. Then combining the above, we see that $\theta_\phi(1,a)$ is bounded by a multiple of
\begin{align}\label{est}
& \sum_{r\in \bb{X}_{\mathcal{O}_S}}\prod_{v\in S_\textrm{fin}}\left((q_v t)\cdot\phi'_v\right)(r) \cdot\prod_{v\in S_\infty} \left((q_vt)\cdot\omega(a_v)\phi_v\right)(r)\\
\leq & \prod_\lambda \left(\sum_{r\in \mathcal{O}_S^{2N}[t^{-1}]t^{-1}}\prod_{v\in S_\textrm{fin}}\left((q_v t)\cdot\phi''_v\right)(r)\cdot\prod_{v\in S_\infty} \left((q_vt)\cdot\phi''_v\right)(\lambda(a)r)\right)^{2nm_\lambda},\nonumber
\end{align}
where $\phi''_v$, $v\in S_\textrm{fin}$ is the characteristic function of $\varpi_v^{k_v}\mathcal{O}_v^{2N}[t^{-1}]t^{-1}$, and $\phi''_v$, $v\in S_\infty$ is the Gaussian function such that
\[
\left((q_vt)\cdot\phi''_v\right)(r)=\exp\left(-C\sum^\infty_{j=1}|r_j|_v^2|q_v|^{2j}\right)
\]
for $r=\sum_j r_j t^{-j}\in \bb{X}_v$. As in \cite{LZ}, applying an equality which invokes Arakelov divisors due to van der Geer and Schoof \cite{GS},  we obtain that
(\ref{est}) is bounded by
\begin{equation}\label{est1}
\prod_\lambda \left(\prod^\infty_{j=1}\left(1+\beta \exp\left(-C_1 \lambda(a) |q|^{\frac{2}{d}j}\right)\right)\right)^{2nm_\lambda},
\end{equation}
where $d=[F:\mathbb{Q}]$, $\beta>0$ only depends on $C,F$, and $C_1>0$ depends on $C, F$ and $\phi$. As an elementary exercise, one can show that for constants  $\rho>1$
and $ \beta>0$, the function
\[
\vartheta(x):=\prod^\infty_{j=1}(1+ \beta e^{-x\rho^j}),\quad x>0
\]
grows at a rate  $O(x^{-\beta/\log \rho})$ as $x\to 0$. Applying this fact, (\ref{est1}) is bounded by a multiple of
\[
\prod_\lambda \sup \left(1, \lambda(a)^{-d \beta/2\log |q|}\right)^{2nm_\lambda}.
\]
Plugging in this estimate into the integral (\ref{est0}), one can easily deduce the desired convergence criterion with $c:=d\beta$, in a way similar to the finite dimensional case \cite{W2}.
\end{proof}

We remark that the upper bound $O(x^{-\beta/\log\rho})$ for the function $\vartheta(x)$ in the above proof may not be optimal, but one can show that in fact $\vartheta(x)$ is also bounded below by a multiple of $x^{-\beta'/\log\rho}$ for some constant $\beta'>0$ as $x\to 0$. Thus a more precise estimate for $\vartheta(x)$ does not affect the essence of the convergence criterion. Since we have established the polynomial growth of the theta function, the following result is immediate from the rapid decay of cusp forms.

\begin{proposition}\label{prop3.2}
The theta integral (\ref{thetaf}) is  absolutely convergent for any cusp form $f$ on $G'(\mathbb{A})$ and $\phi\in 
\mathcal{S}'(\mathbb{X_A})$.
\end{proposition}

As we mentioned previously, Theorem \ref{weil} has direct application to the loop Siegel-Weil formula. To this end we need to further impose certain conditions on the Schwartz function $\phi\in\mathcal{S}'(\bb{X}_{\mathbb{A}})_q$, following \cite{GZ2}.  Let us give the details for completeness. Recall from Section \ref{2.1} that for each place $v$ we have the local Heisenberg group $H_v$ acting on $\mathcal{S}(\bb{X}_v)$. For a finite place $v$, let $\varpi_v^{n_v}\mathcal{O}_v$ be the conductor of the local character $\psi_v$, and for some  integer $k$ which is large enough let
\[
\mathcal{S}_{v,k}=\left\{\phi\in \mathcal{S}(\bb{X}_v): \phi\textrm{ is invariant under }\varpi_v^k\mathcal{O}_v^{2N}((t))\subset H_v\right\}.
\]
Then $\phi\in\mathcal{S}_{v,k}$ can be identified with a function on $(\varpi_v^{n_v-k}\mathcal{O}_v/\varpi_v^{k}\mathcal{O}_v)^{2N}[t^{-1}]t^{-1}$. Under this identification, define
\[
\mathcal{S}_{v,k}^\textrm{fin}=\left\{\phi\in\mathcal{S}_{v,k}: \hat{\phi}\textrm{ has finite support}\right\},
\]
where $\hat{\phi}$ is the Fourier transform of $\phi$. Introduce a subset $\mathcal{E}(\bb{X}_{\mathbb{A}})^\textrm{fin}$ of $\mathcal{E}(\bb{X}_{\mathbb{A}})$, which consists of functions
$\prod_v\phi_v$ such that (1) $\phi_v=\phi_{0,v}$ for each infinite place $v$; (2) $\phi_v=\phi_{0,v}$ for almost all finite places $v$; (3) each remaining component $\phi_v$ is in
$\mathcal{S}_{v,k}^{\textrm{fin}}$ for some $k$. Then we define
\[
\mathcal{S}'(\bb{X}_{\mathbb{A}})^\textrm{fin}_q=(qt)\cdot \textrm{Aut}^0\mathbb{A}\langle t\rangle \cdot \mathcal{E}(\bb{X}_{\mathbb{A}})^\textrm{fin}.
\]

In \cite{GZ2}, it was shown that for $\phi\in \mathcal{S}'(\bb{X}_{\mathbb{A}})^\textrm{fin}_q$, the Eisenstein series
\[
E_\phi(g)=\sum_{\gamma\in G(F[[t]])\backslash G(F((t)))}(\omega(\gamma g)\phi)(0)
\]
converges absolutely when $m>6n+2$ and therefore is an automorphic function on $\widetilde{G}(\mathbb{A}\langle t\rangle)$. See also \cite{L} for the convergence range of a general loop Eisenstein series induced from maximal parabolic subgroups. The main result of \cite{GZ2} together with \cite{LZ} states that if $V$ is anisotropic with $m>6n+2$, then the Siegel-Weil formula
\[
E_\phi(g)=\theta_\phi(g)
\]
holds for any $\phi\in  \mathcal{S}'(\bb{X}_{\mathbb{A}})^\textrm{fin}_q$. Recall that $\theta_\phi(g)$ is the theta integral (\ref{theta1}).
We comment that $ E_{\omega ( h) \phi } ( g) =   E_{\phi  } ( g) $ and $ \theta_{\omega ( h) \phi } ( g) =   \theta_{\phi  } ( g) $ for $ h \in G'( \mathbb{A} \langle t\rangle_+)$ but not for $ h \in  G'( \mathbb{A} \langle t\rangle)$,
which is very different from the classical case.
  Since we have solved the convergence issue for $\theta_\phi(g)$, the methods in \cite{GZ1, GZ2} extend to general orthogonal groups and thus strengthen the Siegel-Weil formula as follows.

\begin{theorem} \label{swf}
Assume that $m>6n+2$, and that either $V$ is anisotropic or $m-b> cn/\log|q|+1$, where $c$ is the constant in Theorem \ref{weil}. Then for $\phi\in\mathcal{S}'(\bb{X}_{\mathbb{A}})_q^\mathrm{fin}$ one has
\[
E_\phi(g)=\theta_\phi(g),
\]
where both sides are absolutely convergent.
\end{theorem}

Finally we remark that in order to extend the Siegel-Weil formula beyond the convergence range, on one hand one has to establish the meromorphic continuation of loop Eisenstein series, which essentially should follow from Garland's work on the Maass-Selberg relations; on the other hand one has to regularize the theta integral on loop orthogonal groups, for which we expect that the affine Satake isomorphism and the Hecke algebras for $p$-adic loop groups studied in \cite{BKP} should play an important role. However according to the last assertion of Theorem \ref{weil}, for $|q|$ large enough the theta integral is always absolutely convergent (excluding the case $m=2$ and $V$ isotropic), which suggests that the Eisenstein series should have meromorphic continuation at least for large $|q|$. Hopefully these issues will be addressed somewhere else.

\section{Rallis constant term formula: anisotropic case}\label{s4}

In this and the next section, we compute the constant terms of the theta lifting $\theta_\phi^f(g)$ along the unipotent radicals of various standard parabolic subgroups of $\widetilde{G}(\mathbb{A}\langle t\rangle)$, which satisfy the so-called tower property. Denote the tower of symplectic groups by $G_j=\rm{Sp}_{2j}$.
Our main result implies that, if the first occurrence of the theta lift of $\pi$ to the tower $\widetilde{G}_{j}(\mathbb{A})$, $j\geq 0$ shows up at $\widetilde{G}_{n}(\mathbb{A})$, then the lift of $\pi$ to each $\widetilde{G}_j(\mathbb{A}\langle t\rangle)$, $j<n$ is {\it cuspidal} in the sense that  $\theta^f_\phi(g)$, $f\in\pi$ has zero constant term along the unipotent radical of every {\it standard} maximal parabolic subgroup. On the other hand the lift to $\widetilde{G}_{j}(\mathbb{A}\langle t\rangle)$ for $j\geq n$ is not cuspidal.

These are along a line parallel to the classical result that the first occurrence of theta lift of an irreducible cuspidal automorphic representation is cuspidal and irreducible. The classical tower property was observed by S. Rallis \cite{R1}. For results on the irreducibility of first occurrence, see \cite{M1, M2} for the case $(\textrm{Sp}_{2n}, \textrm{O}_{2m})$, \cite{JS} for the case $(\widetilde{\textrm{Sp}}_{2n}, \textrm{O}_{2m+1})$, and \cite{Wu} for the case of unitary groups.

In this section we consider the case that $V$ is anisotropic, and we handle the general case in the next section. Our presentations have been influenced by the approach in \cite{MVW} and \cite{Wu}. We shall begin with a description of parabolic subgroups of loop groups and their unipotent radicals, as well as their actions through Weil representations.

\subsection{Parabolic subgroups and unipotent radicals} Let us temporarily assume that $F$ is local. The global setting can be done similarly with only minor modifications. Recall that $W$ is a $2n$-dimensional symplectic space over $F$, and that $G=G_n=\rm{Sp}(W)$. Let $W=W^+\oplus W^-$ be a polarization so that $W^\pm$ are Lagrangian subspaces of $W$. The symplectic form $\langle,\rangle_F$ identifies $W^+$ with $W^{-*}$, the $F$-linear dual of $W^-$, and vice versa.
Let
\[
0=\ell_0^-\subset \ell_1^-\subset \ell_2^-\subset \cdots\subset \ell_n^-=W^-
\]
be a filtration of $W^-$ such that $\dim \ell_i^-=i$, $i=1,\ldots, n$. Then the stabilizer of this filtration is a Borel subgroup $B_o$ of $\textrm{Sp}_{2n}$. We shall also fix a filtration
\[
0=\ell_0^+\subset \ell_1^+\subset \ell_2^+\subset \cdots \subset \ell_n^+=W^+
\]
such that the symplectic form $\langle,\rangle_F$ identifies $\ell_i^+$ with the linear dual of $\ell_i^-$.

Recall the polarization $W((t))=X\oplus Y=W[t^{-1}]t^{-1}\oplus W[[t]]$. For the symplectic loop group $\widetilde{G}(F((t)))=\widetilde{\textrm{Sp}}_{2n}(F((t)))$, the standard Borel subgroup $B$ containing $B_o$ is defined as the subgroup which stabilizes the filtration
\[
tY\subset \ell_1^-+tY\subset \ell_2^-+tY\subset\cdots \subset \ell_n^-+tY.
\]
The subgroup stabilizes $\ell_i^-+tY$ is the maximal parabolic subgroup $P_i$ of $\widetilde{G}(F((t)))$, which corresponds to removing the node $\alpha_i$ from the Dynkin diagram (\ref{dynkin}), where $i=0,1,\ldots, n$. Then the Levi subgroup of $P_i$ is isomorphic to
$\widetilde{G}_{i}(F)\times G_{n-i}(F)$. More explicitly, note that $\ell_i^+\oplus t^{-1}\ell_i^-$ is a symplectic space by restriction of the symplectic form $\langle,\rangle$ given by (\ref{form}),  and $\widetilde{G}_{i}(F)$ is identified with the metaplectic cover of its symplectic group. The second factor $G_{n-i}(F)$ is identified with $\textrm{Sp}((\ell_i^-)^\perp/\ell_i^-)$.
We remark that from the symbol (\ref{symbol}) it is clear that the metaplectic loop group $\widetilde{G}(F((t)))$ splits over $\textrm{Sp}((\ell_i^-)^\perp/\ell_i^-)$ but does not split over
$\textrm{Sp}(\ell_i^+\oplus t^{-1}\ell_i^-)$. In particular $P_0\cong G(F[[t]])\times\mathbb{C}^\times$ splits.

To make the notations less cumbersome, from now on we fix one maximal parabolic subgroup $P$  and remove the index $i$. Then the previous discussion can be  reformulated slightly as follows. Let us assume that there is a decomposition
\[
W=\ell^+\oplus W_0\oplus \ell^-
\]
where $W_0$ is a symplectic subspace and $W_0^\perp=\ell^+\oplus \ell^-$ is a polarization. Let $P=P_a$ be the maximal parabolic subgroup of $\widetilde{G}(F((t)))$ which stabilizes $\ell^-\oplus tY$, where $a=\dim\ell^\pm\leq n$. Then the Levi subgroup of $P$ is
\begin{equation}\label{levi}
L_P=\widetilde{\textrm{Sp}}(W_\ell)\times \textrm{Sp}(W_0),
\end{equation}
where
\begin{equation}\label{wl}
W_\ell:=\ell^+\oplus t^{-1}\ell^-.
\end{equation}
If we decompose $\ell^-\oplus tY$ as
\begin{equation}\label{decom}
t\ell^+[[t]]\oplus tW_0[[t]]\oplus \ell^-[[t]]
\end{equation}
and let $P$ act from the right, then the unipotent radical $U=U_P$ of $P$ has a triangular decomposition
\begin{equation}\label{tri}
U=U^+U^0U^-,
\end{equation}
where
\begin{align}
\label{u+} & U^+=\left\{ n^+(\mu,\beta)=\begin{pmatrix} 1 & \mu & \beta-\mu\mu^*/2  \\ 0 & 1 & -\mu^* \\ 0 & 0 & 1 \end{pmatrix}: \begin{array}{l}\mu\in \textrm{Hom}_F( \ell^+, W_0)[[t]] \\ \beta\in \textrm{Hom}_F(\ell^+, \ell^-)[[t]] \\\beta+\beta^*=0\end{array}  \right\},\\
\label{u-} & U^-=\left\{ n^-(\nu,\gamma)=\begin{pmatrix} 1 & 0 & 0  \\ \nu & 1 & 0 \\ \gamma-\nu^*\nu/2 & -\nu^* & 1 \end{pmatrix}: \begin{array}{l}\nu\in \textrm{Hom}_F(W_0, \ell^+)[[t]]t \\ \gamma\in \textrm{Hom}_F(\ell^-, \ell^+)[[t]]t^2 \\ \gamma+\gamma^*=0\end{array}  \right\},\\
\label{u0} & U^0=\left\{ n^0(\alpha, \delta)=\begin{pmatrix} \alpha & 0 & 0 \\ 0 & \delta & 0 \\ 0 & 0 & \alpha^{*-1}\end{pmatrix}: \begin{array}{l} \alpha\in N_{\textrm{GL}(\ell^+[[t]])}  \\
\delta\in N_{\textrm{Sp}(W_0[[t]])}\end{array}\right\}.
\end{align}

Let us explain the meaning of these notations, and give some remarks.
In the above, $\mu^*\in\textrm{Hom}_F(W_0,\ell^-)[[t]]$ is the dual map of $\mu$ such that
\[
\langle t^{-n} w, x\cdot\mu\rangle=\langle w\cdot\mu^*, t^{-n}x\rangle
\]
for any $w\in W_0$, $x\in \ell^+$ and $n>0$. Apparently the last equation determines $\mu^*$ uniquely, thanks to the non-degeneracy of the symplectic paring. The maps $\beta^*$, $\gamma^*$, $\nu^*$ and $\alpha^*$ are interpreted similarly. Note that the conditions for $\beta$ and $\gamma$ can be also rephrased as
\begin{equation}\label{sym}
\beta\in \left(\textrm{Sym}^2\ell^-\right)[[t]],\quad \gamma\in \left( \textrm{Sym}^2\ell^+\right)[[t]]t^2.
\end{equation}
The group $\textrm{GL}(\ell^+[[t]])$ resp.
$\textrm{Sp}(W_0[[t]])$ is the $F[[t]]$-points of $\rm{GL}(\ell^+)$ resp. $\rm{Sp}(W_0)$, and $N_{\textrm{GL}(\ell^+[[t]])}$ resp. $N_{\textrm{Sp}(W_0[[t]])}$ is the unipotent radical, which consists of elements congruent to 1 modulo $t$.
The above description of $U$, and in particular the conditions on the order of $t$ in (\ref{u+}) and (\ref{u-}), can be verified using the fact that elements of $U$  are symplectic $F((t))$-linear maps which stabilize not only $\ell^-\oplus tY$ but also $(\ell^-\oplus tY)^\perp=\ell^+[[t]]\oplus t^{-1}W_0[[t]]\oplus t^{-1}\ell^-[[t]]$. The details are basically linear algebra calculations, which are left as exercises to the interested readers.  We should mention that thanks to a general PBW argument, the decomposition of $U$ into above subgroups can be given in any prescribed order, and is unique once the order is fixed. This fact will be implicitly used in our manipulation of constant terms below without further comments.
One last remark is that we have regarded $U$ as a subgroup of $P$, which is legitimate because the metaplectic cover splits over $U$.

\subsection{The action of unipotent radical} We continue to assume that $F$ is local.
As before we use the polarization
\[
\bb{W}((t))=\bb{X}\oplus \bb{Y}=\bb{W}[t^{-1}]t^{-1}\oplus \bb{W}[[t]],
\]
and let $\phi\in \mathcal{S}(\bb{X}_{F})$. Apparently $c_u=0$ for each $u\in U$ and  we recall  that by (\ref{gamma0}),
\[
\omega(u)\phi(r)=\psi\left(\frac{1}{2}\langle r \alpha_u, r b_u\rangle\right)\cdot \phi(ra_u),\quad r\in \bb{X}_{F}.
\]
To spell it out, if we write
$r=\sum_i w_i\otimes v_i$, where $w_i\in W[t^{-1}]t^{-1}$ and $v_i\in V$, then
\[
\omega(u)\phi(r)=\psi\left(\frac{1}{2}\sum_{i,j}\langle w_i a_u, w_j b_u\rangle (v_i, v_j)\right)\cdot \phi\left(\sum_i w_i a_u\otimes v_i\right).
\]
For later use let us further specify the actions of various elements in the unipotent radical $U$.
\begin{proposition} \label{uact}
Write $r=x+y+z\in \bb{X}_{F}$, where $x\in \ell^+[t^{-1}]t^{-1}\otimes V$, $y\in W_0[t^{-1}]t^{-1}\otimes V$ and $z\in \ell^-[t^{-1}]t^{-1}\otimes V$. Then we have
\begin{align*}
& \omega(n^+(\mu,0))\phi(r)= \psi\left( \langle y, x\mu\rangle+\frac{1}{2}\langle x\mu, p_{\bb Y}(x\mu)\rangle\right)\cdot \phi\left(r+p_{\bb X}(x\mu-x\mu\mu^*/2-y\mu^*)\right),\\
& \omega(n^+(0,\beta))\phi(r)=\psi\left(\frac{1}{2}\langle x, x\beta\rangle\right)\cdot \phi(r+p_{\bb X}(x\beta)),\\
& \omega(n^-(\nu,0))\phi(r)= \psi\left(\langle z, y\nu\rangle+\frac{1}{2}\langle z\nu^*, p_{\bb Y}(z\nu^*)\rangle\right)\cdot \phi\left(r+p_{\bb X}(y\nu-z\nu^*\nu/2-z\nu^*)\right),\\
& \omega(n^-(0,\gamma))\phi(r)=\psi\left(\frac{1}{2}\langle z, z\gamma\rangle\right)\cdot \phi(r+p_{\bb X}(z\gamma)),\\
& \omega(n^0(\alpha,1))\phi(r)= \psi\left(\frac{1}{2}\langle x\alpha, p_{\bb Y}(z\alpha^{*-1})-p_{\bb X}(z\alpha^{*-1})\rangle\right)\cdot\phi\left(y+p_{\bb X}(x\alpha+z\alpha^{*-1})\right),\\
& \omega(n^0(1,\delta))\phi(r)=\psi\left(\frac{1}{2}\langle p_{\bb X}(y\delta) , p_{\bb Y}(y \delta )\rangle\right)\cdot \phi\left(x+z+p_{\bb X}( y\delta )\right),
\end{align*}
where $p_{\bb X}$ and $p_{\bb Y}$ are the natural projections from $\bb{W}((t))$ to ${\bb X}$ and ${\bb Y}$.
\end{proposition}

The proof is straightforward, making use of (\ref{u+}--\ref{u0}) as well as the relation (\ref{rel}).

\subsection{Tower property} We now turn to the global setting.  %To reduce the notations, we often omit writing $(F)$ or $(\mathbb{A})$ for taking rational or adelic points, as long as the situation is self-explanatory and no confusion arises.
Let us give the main result of this paper and its proof assuming that $V$ is anisotropic. We have the following Rallis constant term formula for the theta lifting, from which the tower property follows.

\begin{theorem}\label{main}
For $\phi\in \mathcal{S}'(\bb{X}_{\mathbb{A}})$ and a cusp form $f$ on $G'(\mathbb{A})$, the constant term of the theta lifting $\theta^f_\phi(g)$ along $U$ is given by
\[
\int_{U(F)\backslash U(\bb{A})}\theta^f_\phi(ug)du=\int_{G'(F)\backslash G'(\bb{A})}\left(\sum_{r\in t^{-1}\ell^-\otimes_F V}\omega(g, h)\phi(r)\right)f(h)dh=\theta^{\ell, f}_{\omega(g)\phi}(1),
\]
where $\theta^{\ell}$ denotes the classical theta lifting from $G'$ to $\widetilde{\mathrm{Sp}}(W_\ell)$, and $\omega(g)\phi$ is interpreted as a Schwartz function in
$\mathcal{S}(t^{-1}\ell^-\otimes V(\mathbb{A}))$ by restriction.
\end{theorem}

As usual, the integral defining the constant term converges because the integration is taken over a compact domain $U(F)\backslash U(\bb{A})$. Note that $t^{-1}\ell^-\otimes_F V$ is a Lagrangian subspace of $W_\ell\otimes_F V$, hence the above constant formula yields the required theta lifting $\theta^{\ell}$. Indeed by restriction to $g\in \widetilde{\mathrm{Sp}}(W_\ell)$, one has $\theta^{\ell, f}_{\omega(g)\phi}(1)=\theta^{\ell, f}_{\phi}(g)$.  Then we obtain the following immediate corollary.

\begin{corollary} \label{cor}
Let $\pi$ be an irreducible cuspidal automorphic representation of $G'(\mathbb{A})$. Assume that the first nonvanishing theta lifting of $\pi$ to $\widetilde{G}_j(\mathbb{A})$ occurs at $j=n$. Then the theta lifting of $\pi$ to each $\widetilde{G}_j(\mathbb{A}\langle t\rangle)$, $j<n$ is cuspidal, i.e. the functions $\theta^f_\phi(g)$, $f\in\pi$ have zero constant term along the unipotent radical of each standard proper parabolic subgroup.
\end{corollary}

For the same reason, since the lift of $\pi$ to $\widetilde{G}_{n}(\mathbb{A})$, which is nonvanishing, occurs in the constant terms of the lifting to $\widetilde{G}_{j}(\mathbb{A}\langle t\rangle)$, $j\geq n$, we conclude that the latter is not cuspidal.
In the rest of this section we assume that $V$ is anisotropic and prove Theorem \ref{main}. In this case we do not need to use the cuspidality of $f$, except for the convergence issue which was already established in Section \ref{s3}. The proof of Theorem \ref{main} for general orthogonal groups will be presented in Section \ref{s5}.

Note that $U(\mathbb{A})$ acts on ${\bb X}_{\mathbb{A}}$ from the right through $u\mapsto a_u$. For $r\in {\bb X}_{F}$ let $U_r$ be the isotropy group of $r$ in $U$. Then it is easy to show that
\begin{equation}\label{psir}
\psi_r: u\mapsto \psi\left(\frac{1}{2}\langle ra_u, rb_u\rangle \right)=\psi\left(\frac{1}{2}\langle r, rb_u\rangle\right)
\end{equation}
is a character of $U_r(F)\backslash U_r(\bb{A})$, noting that $\psi$ is trivial on $F$. It follows that
\begin{equation}\label{ur}
\int_{U_r(F)\backslash U_r(\bb{A})}\omega(ug)\phi(r)du=\omega(g)\phi(r)\cdot \int_{U_r(F)\backslash U_r(\bb{A})}\psi_r(u)du.
\end{equation}
Recall that the measures on the compact quotients $U(F)\backslash U(\bb{A})$ and $U_r(F)\backslash U_r(\bb{A})$ are normalized such that the total volumes are 1. We begin with the contribution of $r\in t^{-1}\ell^-\otimes_F V$ to the constant term.

\begin{lemma} \label{triv}
If $r\in t^{-1}\ell^{-1}\otimes_F V$,  then $U_r=U$ and $\psi_r$ is trivial hence
\[
\int_{U(F)\backslash U(\bb{A})}\omega(ug)\phi(r)du=\omega(g)\phi(r).
\]
\end{lemma}

\begin{proof}
Putting $x=y=0$, $z\in t^{-1}\ell^-\otimes_F V$ in Proposition \ref{uact}. Since $\nu^*\in\textrm{Hom}(\ell^-, W_0)[[t]]t$, $\gamma\in \textrm{Hom}(\ell^-,\ell^+)[[t]]t^2$, we have
\[
z\nu^*\in W_0[[t]]\otimes V, \quad z\nu^*\nu\in W_0[[t]]t\otimes V,\quad z\gamma\in \ell^+[[t]]t \otimes V
\]
Therefore
\[
p_X(z\nu^*)=p_X(z\nu^*\nu)=p_X(z\gamma)=0\quad\textrm{and}\quad \langle z, z\gamma\rangle=0.
\]
Since $\alpha^{*-1}\equiv 1\textrm{ mod } t$, we have $z\alpha^{*-1}\equiv z\textrm{ mod }  \ell^-[[t]]\otimes V$ hence $p_{\bb X}(z\alpha^{*-1})=z$. We deduce from these facts that for any $u\in U(\mathbb{A})$,
\[
ra_u=r, \quad \langle r, rb_u\rangle=0 \quad\textrm{hence}\quad \omega(ug)\phi=\omega(g)\phi.
\]
Therefore the lemma follows.
\end{proof}

Similarly we have the following result.

\begin{lemma} \label{lemn} If $r\in t^{-1}\ell^-\otimes_F V$, then $N_r=N$ and $\omega(n)\phi(r)=\phi(r)$ for $n\in N(\mathbb{A})$.
\end{lemma}

\begin{proof}
The first assertion is clear. Thus  $ra_n=r$ for $n\in N$ and
\[
\omega(n)\phi(r)=\psi\left(\frac{1}{2}\langle r, rb_n\rangle\right)\cdot \phi(r).
\]
Since $rb_n\in \ell^-\otimes_F V[[t]]$, we have $\langle r, rb_n\rangle=0$ because $\ell^-$ is isotropic.
\end{proof}

Since $f(h)$ is invariant under $N(\mathbb{A})$ as well and the volume of $N(F)\backslash N(\bb{A})$ is 1, we have
\[
\int_{G'(F[[t]])\backslash G'(\bb{A}[[t]])}\sum_{r\in t^{-1}\ell^-\otimes_F V}\omega(g,h)\phi(r)\cdot f(h)dh=\int_{G'(F)\backslash G'(\bb{A})}\sum_{r\in t^{-1}\ell^-\otimes_F V}\omega(g,h)\phi(r)\cdot f(h)dh.
\]
Therefore by Lemma \ref{triv},
to prove Theorem \ref{main} amounts to prove that the contribution of each $r\not\in t^{-1}\ell^-\otimes_F V$ to the constant term is zero, and by (\ref{ur}) it suffices to show that if $r\not\in t^{-1}\ell^-\otimes_F V$, then $U_r$ and $\psi_r$ are non-trivial. We divide this task into three parts according to the summands of $X_{F}$, and we make use of the anisotropic assumption on $V$.

It is convenient to introduce a basis of $W$. Let $e_1,\ldots, e_a$ be a basis of $\ell^+$ and $f_1,\ldots, f_a$ be the dual basis of $\ell^-$ under $\langle,\rangle_F$, where $a=\dim\ell^\pm$. Also let $e_{a+1},\ldots e_n, f_{a+1},\ldots, f_n$ be a symplectic basis of $W_0$. Then we can realize $\frak{sp}(W)$, $\frak{sp}(W_0)$ etc. as matrix algebras using these basis.

\begin{lemma} \label{lemx}
If $r=x+y+z\in \bb{X}_{F}$ as in Proposition \ref{uact} and $0\neq x\in \ell^+[t^{-1}]t^{-1}\otimes_F V$, then $U_r$ and $\psi_r$ are non-trivial.
\end{lemma}

\begin{proof}
Let us write
$x=\sum^{l}_{k=1}t^{-k}x_k$, where $x_k\in \ell^+\otimes_F V$, $k=1,\ldots, l$ such that $x_l\neq 0$.
If
\[
x_l=\sum^a_{i=1} e_i\otimes v_i,\quad  v_i\in V,
\]
then at least one of $v_i$, $i=1,\ldots,a$ is nonzero. Since we have taken dual basis of $\ell^+$ and $\ell^-$, the set of linear maps $\beta_0:\ell^+\to \ell^-$ such that
$\beta_0+\beta_0^*=0$ can be identified with Sym$_a$, the set of symmetric $a\times a$ matrices, or more conceptually the symmetric square tensors Sym$^2\ell^-$. See also (\ref{sym}). Consider $\beta=
t^{2l-1}\beta_0\in t^{2l-1}\textrm{Sym}^2\ell^-(\bb{A})$. Then $x\beta\in t^{l-1}\ell^-[[t]]\otimes V(\bb{A})$ hence $p_{\bb X}(x\beta)=0$, so that $n^+(0,\beta)\in U_r(\bb{A})$ by Proposition \ref{uact}. It is easy to compute that
\[
\langle x, x\beta\rangle=\langle x_l, x_l\beta_0\rangle_\mathbb{A}=\textrm{Tr}(\langle x_l, x_l\rangle_V\cdot \beta_0),
\]
where $\langle x_l, x_l\rangle_V:=( v_i, v_j)_{a\times a}$. Since $V$ is anisotropic, $\langle x_l, x_l\rangle_V\neq 0$ so that
\[
\psi_r: n^+(0,\beta)\to \psi\left(\frac{1}{2}\langle x, x\beta\rangle\right)
\]
restricts to a non-trivial character on the arithmetic quotient of $ t^{2l-1}\textrm{Sym}^2\ell^-\subset U_r$, where we have identified $\beta$ with its image $n^+(0,\beta)\in U_r(\bb{A})$.
\end{proof}

Similarly we have the following result.

\begin{lemma} \label{lemz}
If $r=x+y+z\in \bb{X}_{F}$ as in Proposition \ref{uact} and $z\not\in t^{-1}\ell^-\otimes_FV$, then $U_r$ and $\psi_r$ are non-trivial.
\end{lemma}

\begin{proof}
The proof is similar to that of Lemma \ref{lemx}, and we shall only point out the necessary modifications. In this case we may write $z=\sum^l_{k=1}t^{-k}z_k$ with $z_l\neq 0$, $l\geq 2$.
Then by Proposition \ref{uact} one has $n^-(0,\gamma)\in U_r(\bb{A})$, where $\gamma\in t^{2l-1}\textrm{Sym}^2\ell^+(\bb{A})\subset \textrm{Sym}^2\ell^+(\bb{A})[[t]]t^2$, noting that $2l-1>2$. Again using that $V$ is anisotropic, one can show that
\[
\psi_r: n^-(0,\gamma)\mapsto \psi\left(\frac{1}{2}\langle z, z\gamma\rangle\right)
\]
is a non-trivial character on the arithmetic quotient of $ t^{2l-1}\textrm{Sym}^2\ell^+\subset U_r$.
\end{proof}

Finally, we consider the summand $W_0[t^{-1}]t^{-1}\otimes_F V$ of $\bb{X}_{F}$.

\begin{lemma}\label{lemy}
If $r=x+y+z\in \bb{X}_F$ as in Proposition \ref{uact} and $0\neq y\in W_0[t^{-1}]t^{-1}\otimes_F V$, then $U_r$ and $\psi_r$ are non-trivial.
\end{lemma}

\begin{proof}
Let us introduce the congruence subgroup of $\textrm{Sp}(W_0[[t]])$ of level $l\geq 1$,
\[
\textrm{Sp}(W_0[[t]])_l=\left\{g\in\textrm{Sp}(W_0[[t]]): g\equiv 1\textrm{ mod }  t^l\right\}
\]
so that the unipotent radical $N_{\textrm{Sp}(W_0[[t]])}=\textrm{Sp}(W_0[[t]])_1$. There is an exponential map
\begin{equation}\label{exp}
\exp: t^l\frak{sp}(W_0)\to \textrm{Sp}(W_0[[t]])_l,\quad \exp(\xi)=\sum^\infty_{i=0}\frac{\xi^i}{i!}
\end{equation}
which is a bijection. We may write $y=\sum^l_{k=1} t^{-k}y_k$, where $y_k\in W_0\otimes V$, $k=1,\ldots, l$ such that $y_l\neq 0$. Consider
$\delta=\exp(\xi)\in \textrm{Sp}(W_0[[t]])_{2l-1}(\bb{A})$, where $\xi=t^{2l-1}\xi_0\in t^{2l-1}\frak{sp}(W_0)(\bb{A})$ so that $\xi_0\in \frak{sp}(W_0)(\bb{A})$. Then $y\delta \equiv y\textrm{ mod }W_0[[t]]\otimes V(\bb{A})$ so that $p_{\bb X}(y\delta)=y$ and therefore
$\delta\in U_r$ by Proposition \ref{uact}. Write
\[
y_l=\sum^{n-a}_{i=1}e_{a+i}\otimes v_i +\sum^{n-a}_{i=1}f_{a+i}\otimes v_{n-a+i}.
\]
Then one can compute that
\[
\langle y, p_+(y\delta)\rangle=\langle y_l, y_l\xi_0\rangle_\mathbb{A} =\textrm{Tr}(\langle y_l, y_l\rangle_V\cdot\xi_0 \cdot J),
\]
where
\[
\langle y_l, y_l\rangle_V=(v_i, v_j)_{2(n-a)\times 2(n-a)},\quad J=\begin{pmatrix} 0 & -I_{n-a} \\ I_{n-a} & 0\end{pmatrix}.
\]
Again because $V$ is anisotropic, $\langle y_l, y_l\rangle_V\neq 0$ hence
\[
\psi_r: n^0(1,\delta)\mapsto \psi\left(\frac{1}{2}\langle y,  p_+(y\delta) \rangle\right)
\]
restricts to a  non-trivial character on the arithmetic quotient of $t^{2l-1}\frak{sp}(W_0)\subset U_r$, where we have identified $\xi\in t^{2l-1}\frak{sp}(W_0)(\bb{A})$ with its image $\delta=\exp(\xi)\in U_r(\bb{A})$.
\end{proof}

Combining Lemma \ref{triv}--\ref{lemy} finishes the proof of Theorem \ref{main}, under the assumption that $V$ is anisotropic.

\subsection{Variant and examples}

In the above proof of Theorem \ref{main} for the anisotropic case, we did not make use of the integral over $N$. It turns out that in this case it suffices to take the theta lifting from $G'(\bb{A})$ instead of $G'(\bb{A}[[t]])$. Namely, for a cusp form $f$ on $G'(\bb{A})$  and $\phi\in\mathcal{S}'({\bb X}_{\bb A})$ define
\begin{equation}\label{ftheta}
\tilde{\theta}^f_\phi(g)=\int_{G'(F)\backslash G'(\bb{A})}\theta_\phi(g,h)f(h)dh,\quad g\in \widetilde{G}(\bb{A}\langle t\rangle).
\end{equation}
It will be useful to introduce a notion of {\it level} for functions $\varphi\in \mathcal{E}(\mathbb{X}_{\mathbb{A}})$. By definition, any such $\varphi$ is invariant under a subgroup of the adelic Heisenberg group  of the form
\[
\prod_{v\not\in S, v<\infty}\mathcal{O}_v((t))^{2nm}\times \prod_{v\in S}\varpi_v^{k_v}\mathcal{O}_v((t))^{2nm},
\]
where $S$ is a finite set of finite places, and $k_v\in\mathbb{N}$, $v\in S$. Then we say that $\varphi$ is of level $\{k_v\}_{v\in S}$.

We have the following results.

\begin{proposition}\label{propf}
Assume that $V$ is anisotropic. Then the following hold.

(i) Theorem \ref{main} and Corollary \ref{cor} hold if one replaces $\theta^f_\phi(g)$ by $\tilde{\theta}^f_\phi(g)$ given by (\ref{ftheta}).

(ii) For any $q$ with $|q|$ large enough, there exists $ \varphi_q \in  \mathcal{E}({\bb X}_{\mathbb{A}})$  with level depending on $q$, such that the theta lifting $\tilde{\theta}^f_{\phi_q}(g)$ is nonzero for $\phi_q = ( q t ) \cdot \varphi_q $.
 \end{proposition}

\begin{proof}
The proof of (i) is identical to that of Theorem \ref{main} in the anisotropic case.
  To prove (ii), assume that the first nonvanishing theta lifting of $\pi$ to $\widetilde{G}_j(\mathbb{A})$ occurs at $j=n$. It is known from the classical tower property that the theta lifting of $\pi$ to each $\widetilde{G}_j(\bb{A})$, $j\geq n$ is nonzero.
Now fix $j>0$. Keep all the notations as before and assume that ${\bb W}=W\otimes_FV$ where $W$ is a symplectic space of dimension $2j$. Pick up $N>0$ such that $2jN>n$, and decompose ${\bb X}={\bb W}[t^{-1}]t^{-1}={\bb X}_1\oplus {\bb X}_2$, where
\[
{\bb X}_1={\bb W}t^{-1}+\cdots+{\bb W}t^{-N},\quad {\bb X}_2={\bb W}t^{-N-1}+{\bb W}t^{-N-2}+\cdots.
\]
Then there exist $\varphi_1\in \cal{S}({\bb X}_{1,{\bb A}})$ and $f\in \pi$ such that
\[
\theta^f_{\varphi_1}(e)=\int_{G'(F)\backslash G'({\bb A})}\theta_{\varphi_1}(h)f(h)dh\neq 0,
\]
 where
\[
\theta_{\varphi_1}(h)=\sum_{r\in {\bb X}_1}\omega(h)\varphi_1(r).
\]
Choose $\varphi_2\in \cal{E}({\bb X}_{2,{\bb A}})$ such that
 \[ \lim_{ |q| \to \infty } \theta_{ ( q t ) \varphi_2 } (h)=1 \]
where the convergence is uniform in $h$ over the orthogonal group.
 Put $ \varphi_q =  (q^{-1} t ) \varphi_1 \otimes  \varphi_2 $. Then $ (qt) \varphi_q = \varphi_1 \otimes (qt) \varphi_2 $,
 from which we see that
 \begin{equation}
\tilde{\theta}^f_{ (qt) \varphi_q}(e)    \to \theta^f_{\varphi_1}(e)\neq 0,\quad \rm{as} \quad |q|\to\infty.
\end{equation}
 Therefore $\tilde{\theta}^f_{(qt)\varphi_q}(g)\not\equiv 0$ for $|q|$ large.
\end{proof}

\begin{remark} (i) Recall that cuspidality for loop groups means that the constant terms along all {\it standard} maximal paraoblic subgroups are zero.
 Proposition \ref{propf} also applies for the function fields to construct nonzero cusp forms on loop groups.
 
 (ii) We can describe the level of $\varphi_q$ in Proposition \ref{propf} (ii) explicitly. Let $S_q$ be the set of finite places $v$ for which $q_v\not\in \mathcal{O}_v^\times$, and assume that 
 $\varphi_1\otimes\varphi_2$ in the proof of Proposition \ref{propf} is of level $\{k_v\}_{v\in S}$. Put $S_q'=S_q\cup S$. Then $\varphi_q=(q^{-1}t)\varphi_1\otimes \varphi_2$ is of level 
 \[
 \{k_v':=k_v+N\cdot |\textrm{val}_v(q_v)|\}_{v\in S_q'},
 \]
 where $\textrm{val}_v(\cdot)$ is the normalized valuation on $F_v^\times$, and we set $k_v=0$ for $v\in S_q\setminus S$.
 
 (iii) There is also an obvious notion of {\it level} for automorphic functions on loop symplectic groups, which can be seen below. It follows from the previous remark that the level of the cusp form $\tilde{\theta}^f_{\phi_q}(g)$ produced  in this way also depends on $q$. Indeed, assume that $\varphi_q$ is of level $\{k_v'\}_{v\in S_q'}$ as above. Recall that  
 $\psi=\bigotimes\limits_v\psi_v$ is the non-trivial additive character of $\mathbb{A}/F$ used to define the global Weil representation. Denote by $\varpi_v^{-l_v}\mathcal{O}_v$ the conductor of  $\psi_v$ for a finite place $v$. By enlarging $S_q'$ if necessary, we may assume that $v\not| 2$ and $\psi_v$ is unramified for any finite place  $v$ outside $S_q'$.  Then $\tilde{\theta}^f_{\phi_q}(g)$ transforms by a character under the action of
 \[
 \prod_{v\not\in S_q', v<\infty}K_v\times \prod_{v\in S_q'}K_{v, 2k'_v+l_v},
 \]
 where as before $K_v$ is the ``maximal compact subgroup" of $\widetilde{G}_j(F_v((t)))$, and $K_{v, 2k_v'+l_v}$ denotes the kernel of the composed projection
 \[
 K_v\longrightarrow G_j(\mathcal{O}_v((t)))\longrightarrow G_j\left((\mathcal{O}_v/\varpi_v^{2k_v'+l_v}\mathcal{O}_v)((t))\right).
 \]
 Moreover this character is trivial on $K_v$ for each finite place $v\not\in S_q'$, and also trivial on $K_{v, 2k_v'}$ if $v\in S_q'$, $v\not| 2$ and $\psi_v$ is unramified. 
 \end{remark}

Applying Proposition \ref{propf}, we now give two concrete examples from which we obtain nonzero cusp forms on the loop group of $\rm{SL}_2$.

\begin{ex}
The first one arises from the work of R. Howe and I.I. Piatetski-Shapiro \cite{H-PS}. Let $E/F$ be a quadratic extension, and $\beta$ be the quadratic form on $E$ given by the norm from $E$ to $F$, where we regard $E$ as a 2-dimensional space over $F$. Denote by $\textrm{O}_2(F)$ the isometry group of $\beta$, which is anisotropic. The multiplicative group of elements in $E^\times$ with norm 1 can be identified with $\textrm{SO}_2(F)$, and one has an exact sequence
\[
1\longrightarrow \textrm{SO}_2(F)\longrightarrow\textrm{O}_2(F)\longrightarrow\textrm{Gal}(E/F)\longrightarrow 1.
\]
The sign representation of $\textrm{O}_2(F)$ is defined to be the pull back of the sign representation of $\textrm{Gal}(E/F)\cong\mathbb{Z}/2\mathbb{Z}$. We also have the local analog of the above short exact sequence. It turns out that the non-connectedness of $\textrm{O}_2(F)$ plays an important role in the construction.

Let $\mathbb{A}$ and $\mathbb{A}_E$ be the adele rings of $F$ and $E$ respectively. The dual pair $(\textrm{Sp}_4(\mathbb{A}),\textrm{O}_2(\mathbb{A}))$ acts through the oscillator representation $\omega$ on the space $L^2(\mathbb{A}^2_E)$, and its smooth model $\omega^\infty$ is realized on the subspace $\mathcal{S}(\mathbb{A}^2_E)$ of Bruhat-Schwartz functions. We construct an automorphic character $\epsilon_S$ of $\rm{O}_2(\bb{A})$ as follows.
   Let $S$ be a finite set of places of $F$ with $ | S|$ even, and we assume that $S$ contains a non-split  place. Define
  $  \epsilon_S  =\bigotimes\limits_{v} \chi_v$, where $\chi_v$ is trivial for $v\not\in S$ and $\chi_v= \rm{sgn}_v$ for $v\in S$.  The assumptions on $S$ imply that
  \begin{itemize}
  \item $\epsilon_S$ is trivial on $\textrm{Gal}(E/F)$, so that it is an automorphic character of $\rm{O}_2(\bb{A})$;

  \item
      the theta lifting of $\epsilon_S$
    to $\textrm{SL}_2(\bb{A}) $ is zero and any nonzero lifting  to $\textrm{Sp}_4(\bb{A})$ is a cusp form.
    \end{itemize}

 It is known in \cite{H-PS} that there exists $\phi\in \mathcal{S}(\bb{A}_E^2)$ such that  $\theta_\phi^{\epsilon_S} (g) \not\equiv 0$.
Let us explain the construction and give a formula for $\theta_\phi^{\epsilon_S}(g) $.
  For $ v \notin S$, we take $\phi_v \in {\cal S} ( E_v^2 ) $ to be
  the characteristic function of $ {\cal O}_{E_v}^2 $, and for $ v \in S$ take $\phi_v $ to be a Bruhat-Schwartz function
  that transforms by the sign representation of $ \textrm{O}_2( F_v )$.
  Let $\phi = \bigotimes\limits_v \phi_v $, which is invariant under some open compact subgroup $H$ of  $\textrm{O}_2( {\Bbb A} )$. The integral
  \[  \theta_\phi^{\epsilon_S}  (g)  =  \int_{   \textrm{O}_2( F) \backslash   \textrm{O}_2( {\Bbb A} ) }        \sum_{ r \in E^2 } \omega ( g ) \phi  ( r h )  \epsilon_S ( h ) d h
\]
reduces to a finite sum. Indeed, since $\rm{O}_2(F)$ is anisotropic,  $  \textrm{O}_2( F ) \backslash   \textrm{O}_2( {\Bbb A} ) / H $ is finite and we write
\begin{equation}    \textrm{O}_2( {\Bbb A} ) \label{finite}=  \bigsqcup_{i=1}^l  \textrm{O}_2( F ) \alpha_i  H . \end{equation}
Then we have
\begin{equation}    \theta_\phi^{ \epsilon_S} ( g ) = \sum_{i=1}^l   \sum_{ r \in E^2 } \omega ( g ) \phi  ( r \alpha_i ) \epsilon_S ( \alpha_i  ) \end{equation}
One may choose $\phi_v$'s appropriately such that $\theta_\phi^{\epsilon_S}(g) \not\equiv 0$.

Define $\varphi_1\in \mathcal{S}(\mathbb{A}_E^2t^{-1})$ by $\varphi_1(xt^{-1})=\phi(x)$, $x\in \bb{A}_E^2$, and as in the proof of Proposition \ref{propf} put
\[
\varphi_q=(q^{-1}t)\varphi_1\otimes \varphi_2,
\]
where we fix $ \varphi_2\in \mathcal{S}({\mathbb{A}}_E^2 [ t ^{-1} ]   t^{-2}  )$ invariant under $H$, such that $\lim_{ |q| \to \infty } \theta_{ ( q t ) \varphi_2 } (h)=1$ with uniform convergence in $h\in \rm{O}_2(\bb{A})$.  It is easy to see that
   $ \varphi_q \in \mathcal{E}({\mathbb{A}}_E^2  [ t^{-1} ] t^{-1}  )$. We have the theta lifting
   \[ \
   \tilde{\theta}_{ (qt)\varphi_q }^{\epsilon_S } ( e)
    = \int_{   \textrm{O}_2(F )\backslash    \textrm{O}_2({\Bbb A} )}
        \sum_{ r_1  \in E^2 t^{-1} , r_2 \in E^2 [ t^{-1} ] t^2 } \varphi_1 ( r_1  h ) \cdot\left((qt)\varphi_2 \right)( r_2  h ) \epsilon_S ( h  ) d h .
   \]
Using (\ref{finite}), the above integral can be written as a finite sum
  \[  \tilde{\theta}_{ (qt)\varphi_q }^{\epsilon_S } ( e )    = \sum_{i=1}^l
        \sum_{ r_1  \in E^2 t^{-1} , r_2 \in E^2 [ t^{-1} ] t^2 } \varphi_1( r_1  \alpha_i )  \cdot \left(( q t ) \varphi_2\right) ( r_2  \alpha_i ) \epsilon_S ( \alpha_i  ).
   \]
Since $\sum_{   r_2 \in E^2 [ t^{-1} ] t^2 } \left((qt)\varphi_2 \right) ( r_2  \alpha_i ) \to 1$ as $|q|\to\infty$, we obtain that
 \[
 \lim_{|q|\to\infty} \tilde{\theta}_{ (qt)\varphi_q }^{\epsilon_S } ( e )=\sum_{i=1}^l \sum_{ r_1  \in E^2 t^{-1} }  \varphi_1(r_1\alpha_i)=\theta^{\epsilon_S}_\phi(e)\neq 0,
 \]
  which is the classical theta lifting.  Thus by Proposition \ref{propf},
  $\tilde{\theta}_{ (qt)\varphi_q}^{\epsilon_S } (g) $ is a nonzero cusp form on $\widetilde{\rm{SL}}_2(\bb{A}\langle t \rangle)$ for $|q|$ large enough.

Similar to a remark in \cite{H-PS}, for $F=\mathbb{Q}$ and $E$ an imaginary quadratic extension, some of the forms constructed above should be certain loop analogs of Siegel modular forms, which might be interesting and related to the theory of loop groups of Hilbert-modular type studied by H. Garland in \cite{G1}.
\end{ex}

\begin{ex}
Our second example is based on Yoshida's explicit construction of Siegel modular forms of genus 2 \cite{Y}, which uses the theta lifting from the orthogonal group of some quaternion algebras to $\rm{Sp}_4$. We briefly  recall this construction and refer the readers to \cite{Y} for more details. Let $D$ be a definite quaternion algebra over $\bb{Q}$ of discriminant $d^2\in \bb{Q}^{\times2}$, and $R$ be a maximal order of $D$. For $x\in D$ let $x^*$ be the main involution of $x$, and $N(x)=xx^*$, $Tr(x)=x+x^*$ be the reduced norm and trace respectively. Then $V:=(D,N)$ is a 4-dimensional quadratic space over $\mathbb{Q}$. Define
\[
G'=\{(a,b)\in D^\times\times D^\times: N(a)=N(b)=1\}
\]
which acts on $V$ from the right as isometries, by the formula $\rho(a,b)x=a^*xb$.  Let $X=V\oplus V$ and use the same letter $\rho$ for the diagonal action of $G'$ on $X$. Then $\mathcal{S}(X_{\bb A})$, where $\mathbb{A}={\bb A}_{\bb Q}$, provides the global Weil representation of the dual pair $(\rm{Sp}_4, G')$. As usual, for $\phi\in \mathcal{S}(X_{\bb A})$ and a cusp form $f$ on $G'(\bb{A})$ one has the theta lifting
\begin{equation}\label{sg}
\theta^f_\phi(g)=\int_{G'(\bb{Q})\backslash G'(\bb{A})}\sum_{r\in X}\omega(g,h)\phi(r)\cdot f(h)dh,\quad g\in\rm{Sp}_4(\bb{A}).
\end{equation}

The maximal order $R$ gives the subgroup
$K= K_f\times \bb{H}^\times$ of $D^\times_\bb{A}$, where $K_f=\prod_p R_p^\times$, $R_p=R\otimes_\bb{Z}\bb{Z}_p$ and $\bb{H}$ is the Hamilton quaternion algebra. Fix an embedding $\bb{H}^\times\hookrightarrow\rm{GL}_2(\bb{C})$, and for a nonnegative integer $m$ define a representation $\sigma_m:=\rm{Sym}^m\otimes N^{-m/2}$ of $\mathbb{H}^\times$, where $\rm{Sym}^m$ is the $m$-th symmetric power representation and $N$ is the reduced norm of $\bb{H}^\times$. Let $S(R,m)$ be the space of $\sigma_m$-valued automorphic forms on $D_\bb{A}^\times$ of $K$-type $1_{K_f}\otimes \sigma_m$ and trivial central character, which are also said to be of type $(R,\sigma_m,1)$ in \cite{Y}. Then $S(R,m)=0$ if $m$ is odd, hence one may assume that $m=2n$ is even. For a concrete realization, let $W_n$ be the space of polynomials $P$ on $\mathbb{H}$ such that $P(a+bi+cj+dk)=Q(b,c,d)$ where $Q$ is homogeneous of degree $n$ with complex coefficients, and $W_n^*$ be the subspace of $W_n$ which transforms according to $\sigma_{2n}$.

Assume from now on that $d$ is a prime number $p$. Define a congruence subgroup of $\rm{Sp}_4(\bb{Z})$,
\[
\widetilde{\Gamma}_0(p)=\left\{\begin{pmatrix}a & b \\ c & d\end{pmatrix}\in\rm{Sp}_4(\mathbb{Z}): c\equiv 0\mod p\right\}.
\]
Let $\widetilde{G}_k(\widetilde{\Gamma}_0(p))$ be the space of holomorphic functions on $\frak{H}_2$ of weight $k$ and level $\widetilde{\Gamma}_0(p)$, and $\widetilde{S}_k(\widetilde{\Gamma}_0(p))$ be the subspace of cusp forms, where $\frak{H}_2$ is the Siegel upper half space of genus 2. The results in \S5 and \S6 of \cite{Y} give the theta lifting from $S(R,0)\otimes S(R,2n)$ to $\widetilde{S}_{n+2}(\widetilde{\Gamma}_0(p))$ if $n>0$ and from $S(R,0)\otimes S(R,0)$ to $\widetilde{G}_2(\widetilde{\Gamma}_0(p))$
if $n=0$.

We explain this lifting in details. Let
$D^\times_\bb{A}=\bigcup^H_{i=1}D^\times y_i K$ be a double coset decomposition, where $H$ is the class number of $D$, such that $N(y_i)=1$ and $(y_i)_\infty=1$, $1\leq i\leq H$. Put $e_i=|D^\times\cap y_i K_fy_i^{-1}|$, $1\leq i\leq H$. For $1\leq i,j\leq H$ define a lattice $L_{ij}$ of $D$ by
$
L_{ij}=D\cap y_i \left(\prod_p R_p\right) y_j^{-1},
$
and for $P\in W_n^*$ define the theta series $\theta_{ij,P}(z)$, $z\in\frak{H}_2$ by
\[
\theta_{ij,P}(z)=\sum_{(x,y)\in L_{ij}^2}P(x^*y)\exp\left(2\pi i\cdot \rm{Trace}\left(\begin{pmatrix}
N(x) & Tr(xy^*)/2 \\
Tr(xy^*)/2 & N(y)
\end{pmatrix}
z\right)\right),
\]
where $L_{ij}^2:=L_{ij}\oplus L_{ij}$. For $f_1\otimes f_2\in S(R,0)\otimes S(R,2n)$ define the theta lifting
\[
\theta_P^{f_1\otimes f_2}(z)=\sum^H_{i,j=1}\langle \theta_{ij,P}(z),f_1(y_i)\otimes f_2(y_j)\rangle_{\sigma_{2n}}/e_ie_j,
\]
where $\langle\cdot,\cdot\rangle_{\sigma_{2n}}$ is the inner product on the representation space of $\sigma_{2n}$. Then $\theta^{f_1\otimes f_2}_P(z)\in \widetilde{G}_{n+2}(\widetilde{\Gamma}_0(p))$, and one recognizes that it is the reformulation of (\ref{sg}) in the language of classical modular forms. In particular the integral over  $G'(\bb{F})\backslash G'(\bb{A})$ boils down to a finite sum.

Under the assumptions of \cite[Theorem 7.7]{Y}, if $f_1\in S(R,0)$ is a nonzero eigenform of the Hecke operators $T'(l)$, $l\neq p$ (whose definition will not be recalled here),  then there exists an eigenform $f_2\in S(R,2n)$ of $T'(l)$, $l\neq p$ such that $\theta^{f_1\otimes f_2}_P$ is a nonzero cusp form in $\widetilde{S}_{n+2}(\widetilde{\Gamma}_0(p))$. For the application to loop groups, we only consider the case $n=0$ for simplicity. Then one may take $P=1$ and $\theta_{ij}(z):=\theta_{ij,1}(z)$. Assume that $H\geq 2$. Let $f_1,\ldots,f_H$ be a basis of $S(R,0)$ which are eigenforms of $T'(l)$, $l\neq p$ such that $f_1=1$ is the constant function. Then under the assumptions of \cite[Theorem 7.12]{Y}, there exists $f=f_k$ for some $k\geq 2$ such that
\begin{equation}\label{sg2}
\theta^{f}(z):=\theta_1^{1\otimes f}(z)=\sum^H_{i,j=1}\theta_{ij}(z)f(y_j)/e_ie_j
\end{equation}
is a nonzero cusp form in $\widetilde{S}_2(\widetilde{\Gamma}_0(p))$.

We now turn to the loop group setting. Following the idea in the proof of Proposition \ref{propf} (ii), we shall give some explicit construction of nonvanishing cusp forms on the loop group of $\rm{SL}_2$.  Identify $\mathbb{C}^2[[t]]$ with the complex linear dual of $\mathbb{C}^2[t^{-1}]t^{-1}$, and let $\widetilde{\frak{H}}\subset \rm{Hom}_\bb{C}(\bb{C}^2[t^{-1}]t^{-1}, \mathbb{C}^2[[t]])$ be the Siegel upper half space for $\rm{SL}_2(\bb{R}((t)))$ introduced in \cite{Z1}. Pick up $z_0\in\frak{H}_2$ such that $\theta^f(z_0)$ given by (\ref{sg2}) is nonzero. Take $q\in \mathbb{R}$, $q>1$ and define an element $\tilde{z}_{0,q}\in \widetilde{\frak{H}}$ by
\[
\tilde{z}_{0,q}: \bb{C}^2[t^{-1}]t^{-1}\to \mathbb{C}^2[[t]],\quad \sum_{n>0} x_n t^{-n}\mapsto \sum_{n>0} (x_n\cdot z_0) q^{n-1}t^{n-1},
\]
where $x_n\in\mathbb{C}^2$, $n<0$. Then the theta series
\[
\tilde{\theta}_{ij}(\tilde{z}):=\sum_{r\in L_{ij}^2[t^{-1}]t^{-1}}\exp\left(2\pi i\cdot \langle r, r\tilde{z}\rangle\right)
\]
converges absolutely for any $\tilde{z}\in \rm{SL}_2(\bb{R}((t)))\cdot \tilde{z}_{0,q}$, where the pairing $\langle\cdot,\cdot\rangle$ is induced from the  norm $N$ of $D$ and the pairing between $\bb{C}^2[t^{-1}]t^{-1}$ and $\mathbb{C}^2[[t]]$. In particular one has
\[
\tilde{\theta}_{ij}(\tilde{z}_{0,q})=\sum_{\sum_{n>0}(x_n,y_n)t^{-n}\in L_{ij}^2[t^{-1}]t^{-1}}\exp\left(2\pi i\sum_{n>0}q^{n-1}\rm{Trace}\left(\begin{pmatrix}
N(x_n) & Tr(x_ny_n^*)/2 \\
Tr(x_ny_n^*)/2 & N(y_n)
\end{pmatrix}
z_0\right)\right).
\]
Finally we define
\[
\tilde{\theta}^f(\tilde{z})=\sum^H_{i,j=1}\tilde{\theta}_{ij}(\tilde{z})f(y_j)/e_ie_j, \quad \tilde{z}\in \rm{SL}_2(\bb{R}((t)))\cdot \tilde{z}_{0,q},
\]
which can be regarded as an automorphic form on $\rm{SL}_2(\bb{R}((t)))\times \mathbb{R}_{>1}$, viewing $q>1$ as a parameter. Then by Proposition \ref{propf} and Yoshida's result, $\tilde{\theta}^f$ is a cusp form on $\rm{SL}_2(\bb{R}((t)))$ in the sense that it has zero constant term along each standard maximal parabolic subgroup. It is clear that
\[
\tilde{\theta}^f(\tilde{z}_{0,q})\to \theta^f(z_0)\neq 0\quad \rm{ as }\quad q\to\infty,
\]
hence $\tilde{\theta}^f(z)$ cannot be identically zero.
\end{ex}

This finishes the explicit constructions of two examples of nonzero cusp forms on the loop $\rm{SL}_2$. We end this section by briefly mentioning some other explicitly constructible examples. 

\begin{itemize}

\item In an ongoing joint work  of the second named author with Y. Chen, an example of type II theta lifting to the loop group of $\textrm{GL}_n$ is computed;

\item The Saito-Kurokawa lifting \cite{I, PS} can be realized in terms of  the theta lifting  from $\widetilde{\textrm{SL}}_2$ to 
$\textrm{PGSp}_4\cong \textrm{SO}_5$. By switching the roles of symplectic and orthogonal groups in  Corollary 
\ref{cor}, it should be possible to obtain cusp forms on the loop group of  $\textrm{SO}_3\cong \textrm{PGL}_2$.

\end{itemize}

\section{Rallis constant term formula: general case} \label{s5}

Now we treat a general orthogonal group $G'=\textrm{O}(V)$, in which case the proof is more involved. We shall apply some results about affine Gra\ss mannian and loop group orbits from \cite{GZ1, GZ2}.
The idea is that the theta integral (\ref{thetaf}) can be expressed as a sum of orbital integrals. For $r\in {\bb X}_{F}$ let $\mathcal{O}_r$ be the $G'(F[[t]])$-orbit of $r$. Let
$R$ be a set of representatives of  $G'(F[[t]])$-orbits in ${\bb X}_{F}$. Then we may write
\begin{align}\label{orbint}
\theta^f_\phi(g)& =\sum_{r\in R}\theta^{f,\mathcal{O}_r}_\phi(g):= \sum_{r\in R}\int_{G'(F[[t]])\backslash G'(\bb{A}[[t]])}\sum_{r'\in\mathcal{O}_r}\omega(g,h)\phi(r')\cdot f(h)dh\\
&=\sum_{r\in R} \int_{G'(F[[t]])_r\backslash G'(\mathbb{A}[[t]])}\omega(g,h)\phi(r)\cdot f(h)dh,\nonumber
\end{align}
where $G'(F[[t]])_r$ is the isotropy group of $r$ in $G'(F[[t]])$. Thus it is desirable to study the orbits of loop group action and the corresponding orbital integrals.

\subsection{Loop group orbits}

The classification of $G'(F[[t]])$-orbits (same as $G'(F\langle t\rangle_+)$-orbits) in ${\bb X}_{F}$ was given in \cite{GZ1}.
The space $\bb{X}_F$ can be viewed as an $F[[t]]$-module via the identification $\bb{X}_F\cong \bb{W}((t))/ \bb{Y}_F$, noting that $\bb{Y}_F=\bb{W}[[t]]$ is an $F[[t]]$-module. In particular $\bb{X}_{F}$ is torsion.

 Let $\textrm{Gr}(X_F)$ be the set of all $F[[t]]$-submodules of $X_F=W[t^{-1}]t^{-1}$ which are finite dimensional over $F$. Note that $g\in G(F[[t]])$ acts on $\textrm{Gr}(X_F)$ through its action on $X_{F}$, or equivalently through $g\mapsto a_g$ (see (\ref{matrix})). We will simply write $g$ for its action on $X_{F}$ and $\textrm{Gr}(X_F)$, instead of $a_g$.

 For $M\in \textrm{Gr}(X_F)$, define a map
 \begin{equation}\label{tm}
 T_M: M\otimes_F V, \quad \sum_i w_i\otimes v_i\mapsto \sum_{i,j}(v_i, v_j)w_i\otimes w_j\in S^2_t(M),
 \end{equation}
 where $(,)$ denotes the $F[[t]]$-valued bilinear form on $V[[t]]$ that extends the bilinear form $(,)$ on $V$, and $S^2_t(M)$ is the subspace of symmetric tensors in $M\otimes_{F[[t]]}M$. For $i\in S^2_t(M)$, write $U(i)_F$ for the variety with $F$-points the subset of $T_M^{-1}(i)$ that consists of points at which $T_M$ is submersive (see \cite[Lemma 5.7]{GZ1}), which might be empty.
An element
\begin{equation}\label{r}
r=\sum^l_{i=1} w_i\otimes v_i \in {\bb X}_F=X_F\otimes_F V
\end{equation}
defines an $F[[t]]$-map
\begin{equation}\label{fr}
f_r: V[[t]]\to X_F,\quad f_r(v)=\sum^l_{i=1}(v_i, v)w_i.
\end{equation}
 Then it is clear that $\textrm{Im }f_r\in \textrm{Gr}(X_F)$, and we write $T(r)=T_{\textrm{Im }f_r}(r)$ for short.

We have the following classification theorem.

\begin{theorem}\label{clas} \cite{GZ1}
The $G'(F[[t]])$-orbits in ${\bb X}_{F}$ are in one-to-one correspondence with the set of pairs $M\in\mathrm{Gr}(X_F)$, $i\in S^2_t(M)$ such that $U(i)_F\neq\emptyset$. The correspondence is given by $\mathcal{O}_r\mapsto (\mathrm{Im~}f_r, T(r))$.
\end{theorem}

Let $\textrm{Gr}\left(W((t))\right)$ be the set of Lagrangian subspaces $U$ of $W((t))$ such that $U$ is an $F[[t]]$-submodule commensurable with $Y=W[[t]]$. Then the natural projection
$p_X:W((t))\to X$ defines a map
\begin{equation}\label{pgr}
P_{\textrm{Gr}}: \textrm{Gr}\left(W((t))\right)\to \textrm{Gr}(X_F),\quad U\mapsto p_X(U),
\end{equation}
which is $G(F[[t]])$-equivariant but not surjective.
Recall that $P=P_a$ is the maximal parabolic subgroup of $\widetilde{G}(F((t)))$ stabilizing $\ell^-\oplus tW[[t]]$, and we have chosen symplectic basis $\{e_1,\ldots, e_a, f_1,\ldots, f_a\}$ of $\ell^+\oplus \ell^-$ and $\{e_{a+1},\ldots, e_n, f_{a+1},\ldots, f_n\}$ of $W_0$ respectively.
Let
\begin{equation}\label{p0a}
P_{0,a}=G(F[[t]])\cap P,
\end{equation}
which is a non-maximal standard parabolic subgroup of $G(F((t)))$ corresponding to removing the nodes $\alpha_0$ and $\alpha_a$ from the Dynkin diagram (\ref{dynkin}). Then by definition it is the stabilizer of the two-step filtration
\[
0\subset \ell^-\oplus tY\subset Y.
\]
 The Levi subgroup of $P_{0,a}$ is $\textrm{GL}_a\times\textrm{Sp}(W_0)$, where $\textrm{GL}_a\cong \textrm{GL}(\ell^+)\cong\textrm{GL}(t^{-1}\ell^-)$. The latter isomorphism is induced by the symplectic pairing and can be explicated as $g\mapsto {}^tg^{-1}$ using symplectic basis. We have its Weyl group
\begin{equation}\label{w0a}
\mathcal{W}_{0,a}\cong \mathcal{S}_a \times \mathcal{W}_0,
\end{equation}
where $\mathcal{S}_a$ and $\mathcal{W}_0$ are the Weyl groups of $\textrm{GL}_a$ and $\textrm{Sp}(W_0)$ respectively.

The following lemma characterizes the image of $P_\mathrm{Gr}$, in terms of $(G(F[[t]]), P_{0,a})$-double cosets. The resulting double coset representatives are slightly more complicated than those in \cite[Lemma 6.3]{GZ2}. However the advantage of taking left cosets for such a parabolic subgroup $P_{0,a}$ instead of $G(F[[t]])$ is that the map $P_\textrm{Gr}$ is still $P_{0,a}$-equivariant, and on the other hand $P_{0,a}$ normalizes $U=U_P$ and stabilizes $t^{-1}\ell^-\otimes_F V$. It turns out that these properties will be very useful later.

\begin{lemma}\label{pbruhat}
An element $M\in \mathrm{Gr}(X_F)$ is in the image of $P_\mathrm{Gr}$ if and only if there is $p\in P_{0,a}$ such that
\[
M\cdot p=\mathrm{Span}_{F[[t]]}\left\{ t^{-k_1}f_1,\ldots, t^{-k_\sigma}f_\sigma,  t^{-k_a}e_a,  t^{-k_{a-1}}e_{a-1}, \ldots, t^{-k_{a-\tau+1}}e_{a-\tau+1},  t^{-k_{a+1}}f_{a+1},\ldots, t^{-k_{a+\rho}} f_{a+\rho} \right\},
\]
where $\sigma,\tau\geq 0$, $\sigma+\tau\leq a$, $0\leq \rho\leq n-a$, $k_1\geq \cdots\geq k_\sigma\geq 1$, $k_a\geq\cdots\geq k_{a-\tau+1}\geq 1$ and $k_{a+1}\geq\cdots\geq k_{a+\rho}\geq 1$.
\end{lemma}

\begin{proof} Since this is a variant of \cite[Lemma 6.3]{GZ2} and the proof is similar, we will only give a sketch. The sufficiency is easy. To show the necessity, assume that $M=P_\textrm{Gr}(U)$ for $U\in \textrm{Gr}(W((t)))$. By \cite[Lemma 6.2]{GZ2}, $G(F((t)))$ acts on $\textrm{Gr}(W((t)))$ transitively and the stabilizer of $Y$ is $G(F[[t]])$. Thus we may write $U=Yg$ for some $g\in G(F((t)))$. By the well-known Bruhat decomposition, we have natural bijections
\[
G(F[[t]])\backslash G(F((t)))/ P_{0,a}\cong \mathcal{W}\backslash \widetilde{\mathcal{W}}/ \mathcal{W}_{0,a}\cong Q^\vee/ \mathcal{W}_{0,a},
\]
where we have used that $\widetilde{\mathcal{W}}\cong \mathcal{W}\ltimes Q^\vee$ (see the statement below (\ref{bruhat})). Using (\ref{w0a}) and taking dominant cocharacter representatives for the quotient
$Q^\vee/ \mathcal{W}_{0,a}$, we  can decompose $g$ as
\[
g=p_0 \cdot \textrm{diag}\left(t^{k_1}, \ldots, t^{k_n}, t^{-k_1},\ldots, t^{-k_n}\right)\cdot p,
\]
where $p_0\in G(F[[t]])$, $p\in P_{0,a}$, $k_1\geq\cdots\geq k_a$ and $k_{a+1}\geq \cdots\geq  k_n\geq 0$. Note that $k_1,\ldots, k_a$ can be an arbitrary non-increasing sequence of integers, which may be positive or negative. The required assertion follows easily from this.
\end{proof}

This lemma together with some useful notions from \cite{GZ1} which we now recall,
enable us to obtain more concrete description of the loop group orbits.
A {\it quasi-basis} of a finitely generated $F[[t]]$-module $M$ is a set of nonzero elements $u_1,\ldots, u_l\in M$ which generate $M$ and an $F[[t]]$-linear combination $a_1u_1+\cdots+a_lu_l=0$ if and only if all $a_iu_i=0$. A submodule $L\subset M$ is called a {\it primitive submodule} if it is a direct summand and the natural map $L/tL\to M/tM$ induced from $L\hookrightarrow M$ is injective.
By \cite[Lemma 5.3]{GZ1}, if $w_1,\ldots, w_l$ is a quasi-basis of $\textrm{Im }f_r$, then there exist elements $v_1,\ldots, v_l\in V[[t]]$ such that they form a quasi-basis of a primitive submodule of $V[[t]]$ and $r=\sum^l_{i=1}w_i\otimes v_i$. Thus from Lemma \ref{pbruhat} we obtain the following result.

\begin{corollary}\label{rep}
If $\mathrm{Im~}f_r\in \mathrm{Im~}P_\mathrm{Gr}$, then there exists $p\in P_{0,a}$ such that $r\cdot p$ is of the form
\begin{equation}\label{rg}
\sum^\sigma_{i=1}t^{-k_i}f_i\otimes v_i + \sum^\tau_{i=1}t^{-k_{a-i+1}}e_{a-i+1}\otimes v_{a-i+1} + \sum^\rho_{i=1}t^{-k_{a+i}}f_{a+i}\otimes v_{a+i},
\end{equation}
where  the $k_i$'s are given by Lemma \ref{pbruhat}, and all the $v_i$'s form a quasi-basis of a primitive submodule of $V[[t]]$.
\end{corollary}

\subsection{Vanishing of negligible terms}

For each $M\in \textrm{Gr}(X_F)$ we can define
\[
\theta^{f, M}_\phi(g):=\sum_{r\in R: \textrm{ Im }f_r=M}\theta^{f, \mathcal{O}_r}_\phi(g)=\sum_{r\in R:\textrm{ Im }f_r=M}\int_{G'(F[[t]])_r\backslash G'(\mathbb{A}[[t]])} \omega(g,h)\phi(r)\cdot f(h)dh.
\]
We call an orbit $\mathcal{O}_r$  negligible if $\textrm{Im }f_r\not\in \textrm{Im }P_\mathrm{Gr}$. Then we have the following result analogous to  \cite[Lemma 7.4]{GZ2}, which states that the contribution from the negligible orbits to the theta integral itself is zero.

\begin{lemma}\label{lemerror}
If $M\not\in\mathrm{Im~}P_\mathrm{Gr}$, then $\theta^{f,M}_\phi(g)=0$.
\end{lemma}

The proof is essentially the same as that of the above mentioned lemma in \cite{GZ2}, except that we also need to take care of the cusp form $f(h)$. Thus we will only outline the arguments, despite the length and complexity of the original proof. We mention that the proof only uses the property that $f$ is invariant under $N(\bb{A})$, not the cuspidality.

\begin{proof} It suffices to prove that if $r\in \bb{X}_{F}$ is such that $\textrm{Im }f_r\not\in\textrm{Im }P_\textrm{Gr}$, then
\begin{equation}\label{error}
\int_{G'(F[[t]])_r\backslash G'(\bb{A}[[t]])_r}\psi\left(\frac{1}{2}\langle r, rb_h\rangle\right)\cdot f(h)dh=0.
\end{equation}
Recall from (\ref{ug}) that one has the unipotent radical
\[
N=\{g\in G'(F[[t]]): g(0)=1\}
\]
of $G'(F[[t]])$. There is a bijective exponential map
$\exp: t\frak{g}'[[t]]\to N$ defined in the same way as (\ref{exp}), where $\frak{g}'$ is the Lie algebra of $G'$.
In a way similar to (\ref{fr}), for
\[
r=\sum^l_{i=1}w_i\otimes v_i\in {\bb X}_F=W\otimes_F V[t^{-1}]t^{-1},
\]
define
\[
 f'_r: Y_F\to V[t^{-1}]t^{-1}, \quad f'_r(w)=\sum^l_{i=1}\langle w_i, w\rangle v_i.
\]
Let $M'=\textrm{Im }f_r'$, and similarly to (\ref{tm}) define
\[
 T'(r)=\sum_{i,j}\langle w_i, w_j\rangle v_i\otimes v_j \in \wedge^2_t(M'),
\]
where $\wedge^2_t (M')$ is the subspace of skew-symmetric tensors in $M'\otimes_{F[[t]]} M'$.
 By \cite[Lemma 8.2]{GZ2}, $\textrm{Im }f_r\not\in \textrm{Im }P_\textrm{Gr}$ if and only if $T'(r)\neq 0$.  Let
 \[
 \frak{g}'_{M'}=\{\xi\in t\frak{g}'[[t]]: (M'+V[[t]])\xi\subset V[[t]]\},
 \]
 which is a Lie subalgebra of $t\frak{g}'[[t]]$. By \cite[Lemma 8.4]{GZ2}, there is a natural surjective map
$\frak{g}'_{M'} \to \wedge^2_t(M')^*$,  $\xi\mapsto\bar{\xi}$, where $\wedge^2_t(M')^*$ is the dual space of $\wedge^2_t(M')$. Then the character
\[
\exp(\xi) \mapsto \psi\left(\frac{1}{2}\langle r, rb_{\exp(\xi)}\rangle\right)
\]
of the arithmetic quotient of $\exp\frak{g}'_{M'}$ is non-trivial, because the differential of $\exp(\xi)\mapsto \langle r, rb_{\exp(\xi)}\rangle$ is $\xi\mapsto (T'(r), \bar{\xi})$, which is non-trivial thanks to the above quoted results. But $f(h)$ is invariant under $\exp\frak{g}'_{M'}(\bb{A})\subset N(\bb{A})$, hence (\ref{error}) vanishes.  \end{proof}

\subsection{Constant terms of orbital integrals} We are now ready to prove Theorem \ref{main}. Let $R_0=\left\{r\in R: \textrm{Im }f_r\in\textrm{Im }P_\mathrm{Gr}\right\}$. By Lemma \ref{lemerror} we now have
\begin{equation}\label{orbint2}
\theta^f_\phi(g)=\sum_{r\in R_0} \theta^{f,\mathcal{O}_r}_\phi(g).
\end{equation}
To compute the constant term, we assemble the orbital integrals in (\ref{orbint2}) as follows. The actions of the loop group dual pair commute, hence we have $\mathcal{O}_r\cdot p=\mathcal{O}_{r\cdot p}$ for $r\in R$ and $p\in P_{0,a}$.
By Theorem \ref{clas} we have a well-defined map
\[
R\to \textrm{Gr}(X_F), \quad \mathcal{O}_r\mapsto \textrm{Im }f_r,
\]
which is  $P_{0,a}$-equivariant because it is clear that $\textrm{Im }f_{r\cdot p}=\textrm{Im }f_r\cdot p$. Since the map $P_\textrm{Gr}$ is also $P_{0,a}$-equivariant and the map above restricts to $R_0\to \textrm{Im }P_\mathrm{Gr}$, we see that the group $P_{0,a}$ permutes the orbits $\mathcal{O}_r$, $r\in R_0$. Noting that $U\subset P_{0,a}$, the subspace
\[
\bigcup_{r\in R_0}\mathcal{O}_r\subset \bb{X}_{F}
\]
can be decomposed into $(U, G'(F[[t]]))$-orbits. Let $\widetilde{R}$ be a set of representatives of such orbits, and $\widetilde{\mathcal{O}}_r$ be the corresponding $(U,G'(F[[t]]))$-orbit of $r\in\widetilde{R}$. Then we have the constant term
\begin{equation}\label{cst}
\int_{U(F)\backslash U(\bb{A})}\theta_\phi^f(ug)du=\sum_{r\in \widetilde{R}}I^{\widetilde{\mathcal{O}}_r}_\phi(g,f),
\end{equation}
where
\begin{align}\label{lastint}
I^{\widetilde{\mathcal{O}}_r}_\phi(g,f)&:=\int_{U(F)\times G'(F[[t]])\backslash U(\mathbb{A})\times G'(\mathbb{A}[[t]])}\sum_{r'\in\widetilde{\mathcal{O}}_r}\omega(ug,h)\phi(r')\cdot f(h)dh du\\
&=\int_{\left(U(F)\times G'(F[[t]])\right)_r \backslash U(\mathbb{A})\times G'(\mathbb{A}[[t]])} \omega(ug, h)\phi(r)\cdot f(h)dh du.\nonumber
\end{align}
It is clear that the commuting group actions of $P_{0,a}$ (which contains $U$) and $G'(F[[t]])$ stabilize the space $t^{-1}\ell^-\otimes_F V$ as well as its complement in $\bb{X}_{F}$.
Thus  to finish the proof of Theorem \ref{main}, it is enough to prove the following:

\begin{proposition}\label{last}
 If $r\in\wt{R}\setminus (t^{-1}\ell^-\otimes_FV)$, then $I^{\widetilde{\mathcal{O}}_r}_\phi(g,f)=0$.
\end{proposition}

\begin{proof} For $p\in P_{0,a}$ one has
\[
I^{\widetilde{\mathcal{O}}_{r\cdot p}}_\phi(g,f)=I^{\widetilde{\mathcal{O}}_r}_\phi(pg,f),
 \]
because  $p^{-1}\left(U\times G'(F[[t]])\right)_rp=\left(U\times G'(F[[t]])\right)_{r\cdot p}$ and $P_{0,a}\subset P$ normalizes $U$,  noting also that $\psi$ is trivial on $F$. This relation shows that for the nonvanishing of each individual $I^{\wt{\cal O}_r}_\phi(g,f)$ we may consider  representatives of $\left(P_{0,a},G'(F[[t]])\right)$-orbits, not merely $\left(U, G'(F[[t]])\right)$-orbits. Therefore by Corollary \ref{rep} we may assume that $r$ itself is of the form (\ref{rg}).

Consider the vectors $v_1,\ldots, v_{\sigma}, v_{a},\ldots, v_{a-\tau+1}$, $v_{a+1},\ldots, v_{a+\rho}$ in $V[[t]]$, given by (\ref{rg}). Clearly we may assume that $v_i$'s are polynomials,
\[
v_i=\sum^{k_i-1}_{k=0}v_i^{(k)}t^k \in V[t].
\]
Introduce certain subspaces of $V$:
\[
V_1=\textrm{Span}_F\left\{v^{(k_i-1)}_i: 1\leq i\leq \sigma \right\},\quad
V_-=V_z + V_x + V_y,
\]
where
\[
\left\{ \begin{split}
 & V_z=\textrm{Span}_F\left\{v_i^{(k)}:  1\leq i\leq \sigma,  0\leq k<k_i-1\right\}, \\
 & V_x=\textrm{Span}_F\left\{v^{(k)}_{a-i+1}:  1\leq i\leq \tau, 0\leq k<k_{a-i+1}\right\},\\
 & V_y=\textrm{Span}_F\left\{v^{(k)}_{a+i}: 1\leq i\leq \rho, 0\leq k< k_{a+i}\right\}.
\end{split}\right.
\]
Write accordingly $r=r_1+r_2=r+z+x+y$, where
\[
\left\{ \begin{split}
& r_1=\sum^\sigma_{i=1} t^{-1}f_i\otimes v_i^{(k_i-1)} \in t^{-1}\ell^-\otimes_F V_1, \\
 & z=\sum^\sigma_{i=1} t^{-k_i} f_i\otimes v_i-r_1\in t^{-2}\ell^-[t^{-1}]\otimes_F V_z,\\
&  x=\sum^\tau_{i=1} t^{-k_{a-i+1}}e_{a-i+1}\otimes v_{a-i+1}\in t^{-1}\ell^+[t^{-1}]\otimes_F V_x,\\
&  y=\sum^\rho_{i=1}t^{-k_{a+i}} f_{a+i}\otimes v_{a+i}\in t^{-1}W_0[t^{-1}]\otimes_FV_y.
\end{split}\right.
\]

Suppose that $I^{\widetilde{\mathcal{O}}_r}_\phi(g,f)$ does not vanish. Then we need to prove  that $r\in t^{-1}\ell^{-1}\otimes_F V$, which is equivalent to that $r_2=0$, or that $V_-=0$. Since $U_r\subset \left(U\times G'(F[[t]])\right)_r$, we have in particular
\[
\int_{U_r(F)\backslash U_r(\bb{A})}\omega(ug)\phi(r) du =\omega(g)\phi(r)\int_{U_r(F)\backslash U_r(\bb{A})}\psi_r(u)du\neq 0,
\]
where $\psi_r$ is given by (\ref{psir}).
Applying the arguments in the proof of Lemma \ref{lemx}--\ref{lemy}, and also applying other types of elements of $U$ as in Proposition \ref{uact}, it is not hard to deduce
that
\begin{itemize}
\item $V_-$ is an isotropic subspace of $V$;
\item $V_1$ is orthogonal to $V_-$.
\end{itemize}
For instance a variant of Lemma \ref{lemx}--\ref{lemy} implies that $V_x, V_y, V_z$ are all isotropic; application of suitable elements of the form $n^+(\mu,0)$ likewise implies the orthogonality $V_x\perp V_y$, and so on. The detailed proof will be omitted, which is more or less a duplication of the above mentioned lemmas in Section \ref{s4}.

By Corollary \ref{rep}, $v_i$'s form a quasi-basis of a primitive submodule of $V[[t]]$, hence in particular $v_i^{(0)}$'s are linearly independent over $F$. Since $V_-$ is isotropic, applying a suitable element of $\rm{GL}(V_-[[t]])\subset G'(F[[t]])$ and shrinking $V_-$ if necessary, we may assume that $v_i=v_i^{(0)}\in V_-$. Then $r_2$ is of the form
\begin{equation}\label{r2}
r_2=\sum_i t^{-k_i}w_i\otimes v_i
\end{equation}
where $v_i$'s form a basis of $V_-$, and $w_i\in W$ are linearly independent.

Assume that $V_-\neq 0$, and we shall derive a contradiction. Let $Q$ be the proper parabolic subgroup of $G'$ which stabilizes the isotropic space $V_-$, and $N_Q$ be its unipotent radical. Then $N_Q$ fixes $r_2$. We claim that the integral
\begin{equation}\label{inv}
\int_{U(F)_r\backslash U(\bb{A})}\omega(ug,h)\phi(r) du
\end{equation}
is invariant under $N_Q(\bb{A})$, from which it would follow that $I^{\widetilde{\mathcal{O}}_r}_\phi(g,f)=0$
because $f$ is a cusp form.
Thus it remains to verify the claim. Since $V_1\subset V_-^\perp$, for $n\in N_Q(\bb{A})$ we have
\[
\omega(n)\phi(r)=\phi(rn)=\phi(r+r_n)
\]
for some $r_n\in t^{-1}\ell^-\otimes V_-(\bb{A})$. Recall that $r_2$ is of the form (\ref{r2}). Then it is easy to see that there exists $u\in U(\bb{A})$ such that $ra_u=r+r_n$. The properties of
$V_1$ and $V_-$ ensure that
$\langle ra_u, rb_u\rangle=0.$
Therefore
\[
\omega(u)\phi(r)=\phi(r+r_n)=\omega(n)\phi(r),
\]
which implies the claim that (\ref{inv}) is invariant under $N_Q(\bb{A})$.
 This finishes the proof of the proposition hence Theorem \ref{main}.
\end{proof}

\end{document}